\documentclass[12pt,reqno]{article}

\usepackage[usenames]{color}
\usepackage{amssymb}
\usepackage{amsmath}
\usepackage{amsthm}
\usepackage{amsfonts}
\usepackage{amscd}
\usepackage{graphicx}
\usepackage{enumerate,enumitem}
\setlist[enumerate]{label=\arabic*., itemsep=0.1em}

\usepackage{tikz}
\usetikzlibrary{automata,positioning,arrows.meta,fit,decorations.pathmorphing}
\usetikzlibrary{arrows, automata, shapes, positioning}
\usetikzlibrary{arrows.meta}
\usetikzlibrary{patterns,decorations.pathreplacing}
\usetikzlibrary{calc}

\tikzset{
  initial text={},
  >=Stealth,
  every state/.style={minimum size=18pt, inner sep=2pt},
  node distance=2.2cm
}

\usepackage[colorlinks=true,
linkcolor=webgreen,
filecolor=webbrown,
citecolor=webgreen]{hyperref}

\definecolor{webgreen}{rgb}{0,.5,0}
\definecolor{webbrown}{rgb}{.6,0,0}

\usepackage{color}
\usepackage{fullpage}
\usepackage{float}

\usepackage{graphics}
\usepackage{latexsym}
\usepackage{epsf}
\usepackage{breakurl}
\usepackage{hyperref}
\usepackage[capitalize,nameinlink,sort]{cleveref}

\usepackage{todonotes}
\newcommand{\seqnum}[1]{\href{https://oeis.org/#1}{\rm \underline{#1}}}
\DeclareMathOperator{\val}{val}
\DeclareMathOperator{\freq}{freq}
\DeclareMathOperator{\Lop}{\mathcal{L}}
\DeclareMathOperator{\cL}{\mathrm{Fac}}
\newcommand{\que}{\mathord{?}}
\newcommand{\cA}{\{{\tt 0},{\tt 1}\}}
\newcommand{\N}{\mathbb{N}}
\newcommand{\Z}{\mathbb{Z}}

\begin{document}

\theoremstyle{plain}
\newtheorem{theorem}{Theorem}
\newtheorem{corollary}[theorem]{Corollary}
\newtheorem{lemma}[theorem]{Lemma}
\newtheorem{proposition}[theorem]{Proposition}

\theoremstyle{definition}
\newtheorem{definition}[theorem]{Definition}
\newtheorem{example}[theorem]{Example}
\newtheorem{conjecture}[theorem]{Conjecture}

\theoremstyle{remark}
\newtheorem{remark}[theorem]{Remark}

\begin{center}
\vskip 1cm{\LARGE\bf
Symbols frequencies in the Thue--Morse word in base $3/2$ and related conjectures
}
\vskip 1cm
\large
Julien Cassaigne \\
CNRS,  I2M UMR 7373,  Aix-Marseille Université,  13453 Marseille, France \\
\href{mailto: Julien.Cassaigne@math.cnrs.fr}{\tt Julien.Cassaigne@math.cnrs.fr}

and

Basti\'an Espinoza, Michel Rigo,  Manon Stipulanti\\
Department of Mathematics, University of Li{\`e}ge\\
All{\'e}e de la D{\'e}couverte 12 (B37), 4000 Li{\`e}ge, Belgium\\
\href{mailto: BAEspinoza@uliege.be}{\tt BAEspinoza@uliege.be}, \href{mailto: M.Rigo@uliege.be}{\tt M.Rigo@uliege.be}, \href{mailto: M.Stipulanti@uliege.be}{\tt M.Stipulanti@uliege.be}\\
\end{center}

\vskip .2 in
\begin{abstract}
We study a binary Thue--Morse-type sequence arising from the base-$3/2$ expansion of integers, an archetypal automatic sequence in a rational base numeration system. Because the sequence is generated by a periodic iteration of morphisms rather than a single primitive substitution, classical Perron--Frobenius methods do not directly apply to determine symbol frequencies. We prove that both symbols ${\tt 0},{\tt 1}$ occur with frequency $1/2$ and  we show uniform recurrence and symmetry properties of its set of factors. The proof reveals a structural bridge between combinatorics on words and harmonic analysis: the first difference sequence is shown to be Toeplitz, providing dynamical rigidity, while filtered frequencies naturally encode a dyadic structure that lifts to the compact group of $2$-adic integers. In this $2$-adic setting, desubstitution becomes a linear operator on Fourier coefficients, and a spectral contraction argument enforces uniqueness of limiting densities. Our results answer several conjectures of Dekking (on a sibling sequence) and illustrate how harmonic analysis on compact groups can be fruitfully combined with substitution dynamics.
\end{abstract}


\bigskip


\section{Introduction}
\label{sec:intro}

Given a sequence $\mathbf{s}=(s_n)_{n\ge 0}\in A^{\mathbb{N}}$ over a finite alphabet~$A$, two fundamental questions concern the existence and the computation of the frequency of occurrences of a symbol~$a\in A$ appearing in $\mathbf{s}$, namely  the quantity
\[
  \lim_{N\to \infty} \frac{\# \{ 0\le i<N \mid s_i=a\}}{N}.
\]
Answering these questions has applications in several fields,  e.g.,  invariant measures in symbolic dynamics and ergodic theory \cite{MorseHedlund,Queffelec}, normality in number theory \cite{Becher,SGNT}, structure and balance properties in combinatorics on words \cite{Adamczewski,Lothaire2,MorseHedlund2}, data compression and pattern prediction in information theory \cite{1096090}, densities in aperiodic tilings \cite{BaakeGrimm}.

In this article, our initial motivation was to answer a conjecture that has been open for five years \cite{RigSti1}. It concerns the symbol frequencies in the {\em Thue--Morse word in base $3/2$}:
\begin{equation}\label{eq:t32}
  \mathbf{t}_{3/2}={\tt 001110111110110111110000110110}\cdots,
\end{equation}
an infinite sequence generated by a particular type of substitution obtained by periodically iterating two morphisms of constant length~$3$; see~\cref{ssec:per} for precise definitions. Focusing on this sequence is natural: it is a typical automatic sequence in a rational base (here $3/2$).  Although generated by simple rules, it exhibits a rich and non-trivial structure that reflects the complexity of the object.  We prove that the frequency of {\tt 0} is indeed $1/2$ --- as suggested by numerical computations as pictured in~\cref{fig:freq12} --- but we obtain more: we prove that $\mathbf{t}_{3/2}$ is uniformly recurrent and its set of factors is closed under bit-wise complement and reversal. In contrast with classical primitive substitutive systems, $\mathbf{t}_{3/2}$ arises from a periodically iterated morphism and mixes two scaling mechanisms, preventing a direct application of Perron--Frobenius theory and requiring a different approach to establish the existence of symbol frequencies. Compared with integer base systems or substitutive systems, the language of the base-$3/2$ numeration system is highly non-trivial: any two distinct infinite subtrees are non-isomorphic, which is an important difference to overcome.

\begin{figure}[h!t]
  \centering
\includegraphics[height=5cm]{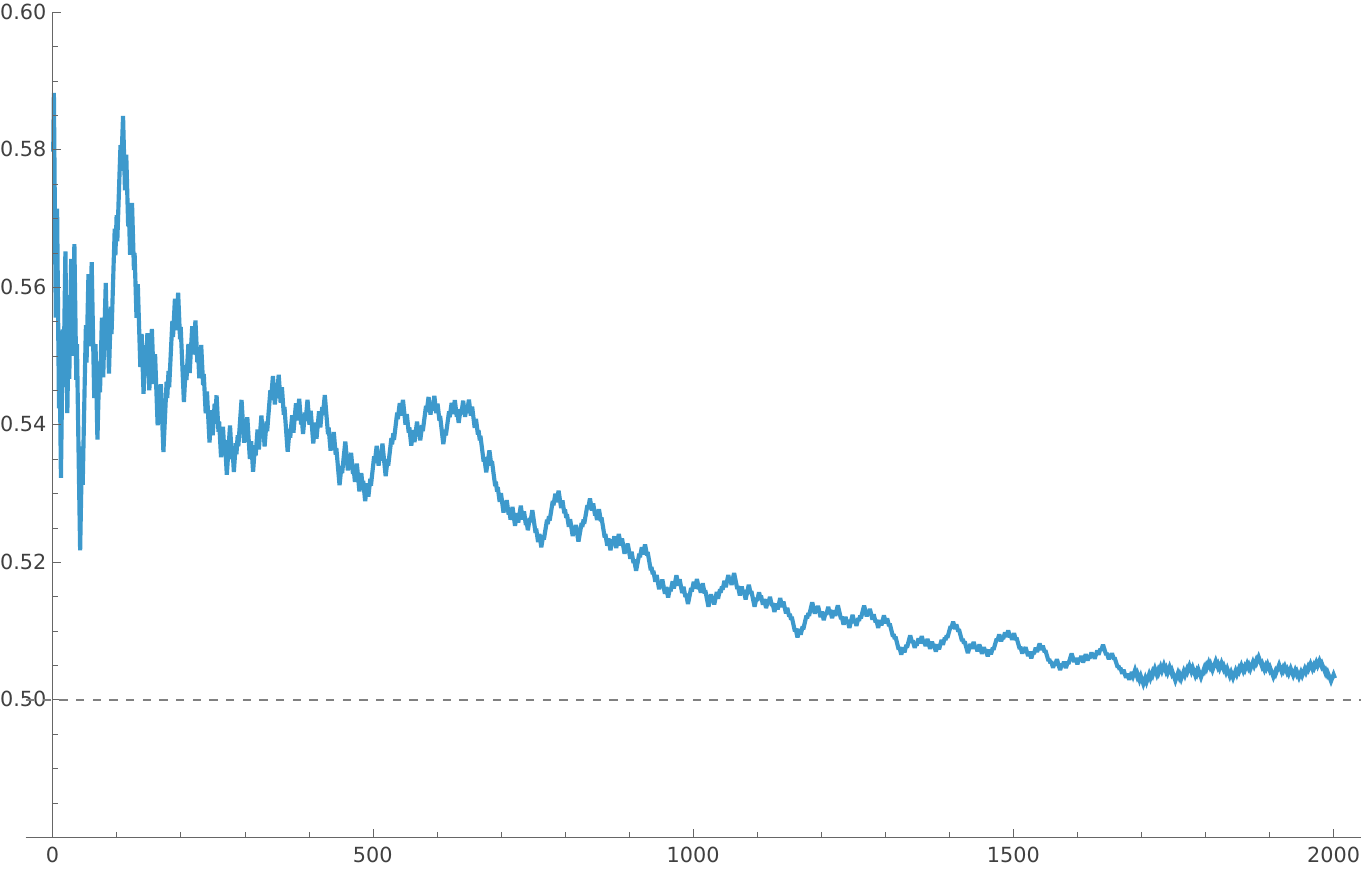} \includegraphics[height=5cm]{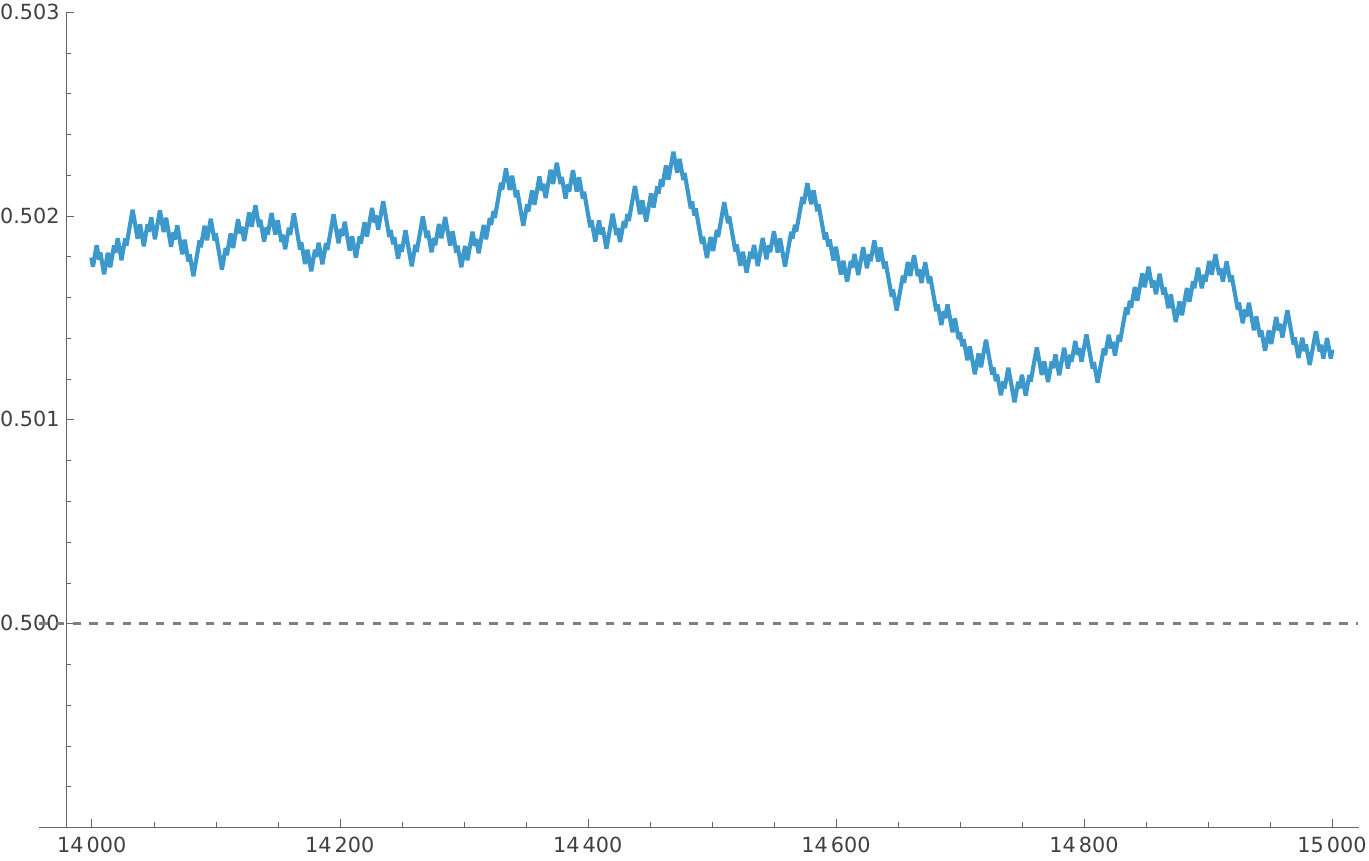} 
  \caption{Estimation of the frequency of {\tt 0} in prefixes of the Thue--Morse word $\mathbf{t}_{3/2}$ in base $3/2$: on the left,  for prefixes with length $\le 2000$; on the right,  those with length in $[14000,15000]$.}
  \label{fig:freq12}
\end{figure}

We also answer several conjectures related to a sibling sequence~$\mathbf{t}'$ introduced by Dekking \cite{DekTM} and recalled in~\cref{ssec:dekk}.  As already mentioned before,  these sequences $\mathbf{t}_{3/2}$ and $\mathbf{t}'$ are linked to the theory of numeration systems.  The symbols of $\mathbf{t}_{3/2}$ can be computed from the expansions of integers in a rational base numeration system \cite{AFS}.  More precisely,  for any integer $n\ge 0$,  the $n$-th symbol of  $\mathbf{t}_{3/2}$ is the parity of the sum-of-digits of $n$ in its base-$3/2$ expansion.  Rational base numeration systems,  introduced by Akiyama, Frougny,  and Sakarovitch in 2008~\cite{AFS} and initially related to a question of Mahler about the distribution of the fractional parts of $\{z\, (3/2)^n\}$,  where $z$ is a real,  have attracted the attention of many researchers \cite{MR4907985,MR4011502,RigSti1,RigSti2,MR4432963}. Our solution using techniques from abstract harmonic analysis suggests possible extensions to similar problems and, in particular, sheds new light on long-standing conjectures related to the Oldenburger--Kolakoski word \cite{Oldenburger,Kolakoski}; see~\cref{ssec:per} for its definition.

The study of substitutions and morphic words lies at the crossroads of combinatorics on words and symbolic dynamics, but also harmonic analysis. Beyond their combinatorial structure, substitutions generate dynamical systems whose spectral properties encode arithmetic and structural information. Tools from abstract harmonic analysis, such as Fourier--Stieltjes transforms, spectral measures, and Riesz products, play a central role in the analysis of these systems \cite{Queffelec}. In particular, the spectral theory of substitutive dynamical systems provides a framework to distinguish pure point, singular continuous, and absolutely continuous behaviors \cite{BaakeGrimm}.

Let us give a few more examples highlighting the interactions between these seemingly distant topics. Baake and Grimm derive a Fourier recursion and functional equation for the Thue--Morse autocorrelation measure, yielding an explicit Riesz product representation and a purely singular continuous spectrum \cite{BaakeGrimmpaper}. Still related to the Thue--Morse word, uniform exponential bounds are obtained for discrete Fourier coefficients of truncated Thue--Morse sums, averaged over arithmetic parameters \cite{MullnerS}. Finally, automatic sequences admit an efficient decomposition into a structured component and a Gowers-uniform component, showing via higher-order Fourier analysis that sequences orthogonal to periodic ones have small Gowers norms and thus behave pseudorandomly with respect to additive patterns \cite{MR4596606}.

The remainder of~\cref{sec:intro} is organized as follows. In \cref{ssec:sub} we briefly recall the notion of substitutions and  fixed points. Then,  in \cref{ssec:per},  the generation of infinite words by substitutions is extended to the case where a finite number of morphisms are applied periodically.  We recall that the famous Oldenburger--Kolakoski sequence can be obtained in this way and list the main conjectures about it. In \cref{ssec:32TM}, we briefly introduce the base-$3/2$ numeration system and the corresponding Thue--Morse sequence~$\mathbf{t}_{3/2}$, followed by the variation proposed by Dekking in \cref{ssec:dekk}. Finally,  in~\cref{sec: contrib}, we talk about the organization of the rest of the paper and we list the main contributions.


\subsection{Substitutions or iterated morphisms}\label{ssec:sub}
For an alphabet $A$,  we let $A^*$ denote the set of finite words over $A$.  Endowed with concatenation product, it is a monoid with the \emph{empty word~$\varepsilon$} as neutral element. 
To distinguish finite and infinite words,  the latter are written in bold.
A map $f:A^*\to A^*$ is a {\em morphism} if it is a homomorphism of monoids, i.e., $f(uv)=f(u)f(v)$ for all $u,v\in A^*$. 
For an infinite word $\mathbf{x}=(x_n)_{n\ge 0}$ and integers $0\le i\le j$,  we let $\mathbf{x}[i,j]$ (resp.,   $\mathbf{x}[i,j)$) denote the factor $x_i x_{i+1}\cdots x_j$ (resp.,   $x_i x_{i+1}\cdots x_{j-1}$) of $\mathbf{x}$. If $i=0$, then $\mathbf{x}[0,j)$ is the  length-$j$ prefix of~$\mathbf{x}$. If $u\in A^*$ is a word and $a$ is a symbol in $A$, we let $|u|_a$ denote the number of occurrences of $a$ in $u$. This notation extends to factors: if $v$ is a finite word, then $|u|_v$ denotes the number of occurrences of $v$ as a factor of $u$.

We begin by recalling the notion of iterated morphisms. 
It is a classical method for generating infinite words. We specify the main definitions and start with the (standard) {\em Thue--Morse sequence} $\mathbf{t}=(t_n)_{n\ge 0}$ (\seqnum{A010060}) starting with
\[
  \mathbf{t}={\tt 0110100110010110}\cdots.
\]
This element of $\{{\tt 0},{\tt 1}\}^\mathbb{N}$ can be defined by $t_n=\mathbf{s}_2(n)\bmod{2}$,  where $\mathbf{s}_2$ is the sum-of-digits function in base $2$.  The Thue--Morse sequence satisfies the relations
\begin{equation}
  \label{eq:tmbase2}
  t_{2n}=t_n \quad\text{ and }\quad t_{2n+1}=1-t_n,\quad \forall n\ge 0.
\end{equation}
 Since every length-$2$ factor $t_{2n}t_{2n+1}$ of $\mathbf{t}$, occurring in an even position, is either ${\tt 01}$ or ${\tt 10}$, it is trivial to see that the frequency of ${\tt 0}$ (and thus  of ${\tt 1}$) is $1/2$. This sequence is an example of $2$-automatic sequences: for all $n\ge 0$, $t_n$ is the output of a deterministic finite automaton fed with the base-$2$ expansion of $n$. It is thus the fixed point of a $2$-uniform morphism,  namely $f:{\tt 0}\mapsto {\tt 01}$ and ${\tt 1}\mapsto {\tt 10}$. Hence,  $\mathbf{t}$ can be obtained by iterating $f$ on ${\tt 0}$. The sequence of finite words $(f^n({\tt 0}))_{n\ge 0}$ converges to $\mathbf{t}$, for the product topology, where $\{{\tt 0},{\tt 1}\}$ has the discrete topology:
  \begin{align*}
    f({\tt 0})&={\tt 01},\\
    f^2({\tt 0})&={\tt 0110},\\
    f^3({\tt 0})&={\tt 01101001},\\
    &\vdots  
  \end{align*}
Indeed, the length $|f^n({\tt 0})|=2^n$ goes to infinity with $n$ and each image $f^n({\tt 0})$ is a prefix of $f^{n+1}({\tt 0})$. For a survey on the Thue--Morse word, see \cite{ubi}. One of its well-known combinatorial features is that it is overlap-free: it contains no factor of the form $auaua$ where $a$ is a symbol and $u\in\{{\tt 0},{\tt 1}\}^*$ is a finite word. For an integer $k\ge 2$,  a sequence is \emph{$k$-automatic} if it is the image under a \emph{coding} (i.e., a letter-to-letter morphism) of a fixed point of a \emph{$k$-uniform} morphism (i.e.,  the length of the image of every letter is $k$).
For references on automatic sequences and combinatorics on words, see \cite{AS,Lothaire2,Rigo_book}. 

Cobham showed in 1972 that for an automatic sequence, if the frequency of a symbol exists, then it is a rational number \cite{Cobham1972}. 
The larger class of morphic sequences is obtained by relaxing the assumption that the morphism generating the word has constant length. In that situation, if the frequency of a symbol exists, then it is an algebraic number \cite[Thm.~8.4.5]{AS}. Using Perron--Frobenius theory for primitive matrices,  frequencies of symbols exist and can be obtained as the normalized Perron eigenvector of the adjacency matrix associated with the morphism,  i.e.,  where the $j$-th column records the number of occurrences of each symbol in the image of the letter $j$.
For example,  the Tribonacci word ${\tt 01020100102}\cdots$ (\seqnum{A080843}) is the fixed point of the morphism ${\tt 0}\mapsto {\tt 01}$, ${\tt 1}\mapsto {\tt 02}$ and ${\tt 2}\mapsto {\tt 0}$ and its adjacency matrix is
\[
  \begin{pmatrix}
    1&1&1\\
    1&0&0\\
    0&1&0\\
  \end{pmatrix}.
\]


\subsection{Periodically iterated morphisms}\label{ssec:per}
One may generalize the construction of morphic sequences by replacing a single morphism with a finite family of several morphisms applied periodically. This construction produces a larger class of infinite words \cite{Lepisto}.
\begin{definition}
\label{def: alternating fixed point}
Let $r\ge 1$ be an integer, let $A$ be a finite alphabet, and let $f_0,\ldots,f_{r-1}$ be $r$ morphisms over $A^*$. An infinite word $\mathbf{x}=(x_n)_{n\ge 0}$ over $A$ is an {\em alternating fixed point} of $(f_0,\ldots,f_{r-1})$ if 
$$\mathbf{x}=f_0(x_0)f_1(x_1)\cdots f_{r-1}(x_{r-1})f_0(x_r)\cdots f_{i\bmod{r}}(x_i)\cdots .$$
In the literature, one also finds the terminology of \emph{periodic iteration of morphisms} \cite{Endrullis,Lepisto}. 
\end{definition}

The famous {\em Oldenburger--Kolakoski word} $\mathbf{k}={\tt 2211212212211}\cdots$ (shift of \seqnum{A000002} where the first {\tt 1} has been conveniently deleted) can be obtained by periodically iterating the two morphisms \cite{CulikLepisto}
\[
  k_0:\left\{\begin{array}{l}
               {\tt 1}\mapsto {\tt 2},\\
               {\tt 2}\mapsto {\tt 2}{\tt 2},\\
             \end{array}\right.
           \quad\text{ and }\quad
k_1: \left\{\begin{array}{l}
               {\tt 1}\mapsto {\tt 1},\\
               {\tt 2}\mapsto {\tt 1}{\tt 1}.\\
            \end{array} \right.
\]
The first few iterations give
\begin{align*}
  k_0({\tt 2})&={\tt 22},\\
  k_0({\tt 2})k_1({\tt 2})&={\tt 2211},\\
  k_0({\tt 2})k_1({\tt 2})k_0({\tt 1})k_1({\tt 1})&={\tt 221121}, \\
  &\vdots 
\end{align*}
It is the unique word $\mathbf{k}$ over $\{1,2\}$ starting with $2$ and satisfying $\mathsf{RL}(\mathbf{k})=\mathbf{k}$, where $\mathsf{RL}$ is the \emph{run-length encoding map}.  It is a challenging object of study in combinatorics on words.  To account for this,  we recall several long-standing conjectures concerning~$\mathbf{k}$ \cite{Sing}.

\begin{itemize}
\item It is conjectured that both letters occur with frequency $1/2$ in $\mathbf{k}$ \cite{Keane91}. The best bounds from Rao~\cite{Rao-Kola} improving on Chv\'atal are
  \[
    0.49992 \le \liminf_N  \frac{|\mathbf{k}[0,N)|_2}{N}
    \le \limsup_N  \frac{|\mathbf{k}[0,N)|_2}{N} \le 0.50008.
  \]

  \item It is unknown whether every factor occurs infinitely often in $\mathbf{k}$ (i.e.,  \emph{recurrence}), or even with bounded gaps (i.e.,  \emph{uniform recurrence}).

  \item The {\em reversal} of a word $a_1\cdots a_n$ is the word $a_n\cdots a_1$,  with $a_i\in A$. It is open whether the set of factors of $\mathbf{k}$ is closed under reversal or complement (exchanging ${\tt 1}$'s and ${\tt 2}$'s) \cite{Kimberling}.

  \item Although the factor complexity of $\mathbf{k}$ is known to be bounded by a polynomial, its precise growth remains unclear (recall that,  for an infinite sequence $\mathbf{x}$,  its \emph{factor complexity} $\mathsf{p}_\mathbf{x}$ gives,  for each integer $n\ge0$,  the number of length-$n$ factors of $\mathbf{x}$).

  \item For Oldenburger--Kolakoski sequences over larger alphabets or with different parameters, analogous questions concerning frequencies, recurrence, and structural properties are largely open.  See~\cite{Jamet24} for a recent probabilistic perspective and~\cite{AJLTC-2024} where pseudo-substitutions are iterated not necessarily in a periodic way as in~\cref{def: alternating fixed point}. 
\end{itemize}

As observed by Dekking \cite{DekkingReg} for the Oldenburger--Kolakoski word $\mathbf{k}$, an alternating fixed point can also be obtained by an $r$-block substitution as defined below.  For any integer $r\ge 1$,  we let $A^r$ denote the set of length-$r$ words over $A$.

\begin{definition}
  Let $r\ge 1$ be an integer and let $A$ be a finite alphabet. An {\em$r$-block substitution} $\beta:A^r\to A^*$ maps a word $w_0\cdots w_{rn-1} \in A^*$ to 
  \[
  \beta(w_0\cdots w_{r-1}) \beta(w_r\cdots w_{2r-1})\cdots \beta(w_{r(n-1)}\cdots w_{rn-1}).
  \] 
  If the length of the word is not a multiple of $r$, then the remaining suffix is ignored under the action of $\beta$. An infinite word $\mathbf{x}=(x_n)_{n\ge 0}$ over $A$ is a {\em fixed point of the $r$-block substitution} $\beta:A^r\to A^*$ if it satisfies $\mathbf{x}= \beta(x_0\cdots x_{r-1}) \beta(x_r\cdots x_{2r-1})\cdots$.
\end{definition}

Let $r\ge 1$ be an integer, let $A$ be a finite alphabet, and let $f_0,\ldots,f_{r-1}$ be $r$ morphisms over $A^*$. It is straightforward to see that if an infinite word over $A$ is an alternating fixed point of $(f_0,\ldots,f_{r-1})$, then it is a fixed point of an $r$-block substitution \cite{RigSti1}.  As an example,  $\mathbf{k}$ is a fixed point of the $2$-block substitution given by 
\[
  \kappa:\left\{
  \begin{array}{l}
    {\tt 11}\mapsto h_0({\tt 1})h_1({\tt 1})={\tt 21},\\
    {\tt 12}\mapsto h_0({\tt 1})h_1({\tt 2})={\tt 211},\\
    {\tt 21}\mapsto h_0({\tt 2})h_1({\tt 1})={\tt 221},\\
    {\tt 22}\mapsto h_0({\tt 2})h_1({\tt 2})={\tt 2211}.\\
  \end{array}\right.
\]


\subsection{The base $3/2$ and the corresponding Thue--Morse word}\label{ssec:32TM}

We recall that, in the base-$3/2$ numeration system, any positive integer $n$ is written
\begin{align}
\label{eq: 32expansions}
  n=\sum_{i\ge 0}d_i\, \frac{1}{2}\left(\frac{3}{2}\right)^i,
\end{align}
with digits $d_i\in\{0,1,2\}$ \cite{AFS}. The \emph{$3/2$-expansion} of $n$ is denoted by $\langle n \rangle_{3/2}$. By convention, the empty word $\varepsilon$ is the $3/2$-expansion of $0$.
See~\cref{tab:base32expansions} for the first few expansions.

\begin{table}[h!]
\centering
\begin{tabular}{|c|r|c|c|}
\hline
$n$ & $\langle n \rangle_{3/2}$ & $\mathbf{s}(n)$ & $t_n$\\
  \hline
0 & $\varepsilon$ & 0 & 0 \\
1 & 2 & 2& 0 \\
2 & 21 & 3&1 \\
3 & 210 & 3&1 \\
4 & 212 & 5& 1\\
5 & 2101 & 4&0 \\
6 & 2120 & 5& 1\\
7 & 2122 & 7& 1\\
8 & 21011 & 5& 1\\
\hline
\end{tabular}
\hspace{0.5cm}
\begin{tabular}{|c|r|c|c|}
\hline
$n$ & $\langle n \rangle_{3/2}$ &  $\mathbf{s}(n)$ & $t_n$ \\
  \hline
9 & 21200 & 5& 1\\
10 & 21202 & 7 & 1 \\
11 & 21221 & 8 & 0 \\
12 & 210110 & 5 & 1 \\
13 & 210112 & 7 & 1 \\
14 & 212001 & 6 & 0 \\
15 & 212020 & 7 & 1 \\
16 & 212022 & 9 & 1 \\
17 & 212211 & 9 & 1 \\
\hline
\end{tabular}
\hspace{0.5cm}
\begin{tabular}{|c|r|c|c|}
\hline
$n$ & $\langle n \rangle_{3/2}$ &  $\mathbf{s}(n)$ & $t_n$  \\
  \hline
18 & 2101100 & 5 & 1 \\
19 & 2101102 & 7 & 1 \\
20 & 2101121 & 8 & 0 \\
21 & 2120010 & 6 & 0 \\
22 & 2120012 & 8 & 0 \\
23 & 2120021 & 8 & 0 \\
24 & 2120020 & 7 & 1 \\
25 & 2120222 & 11 & 1 \\
26 & 2122111 & 10 & 0 \\
\hline
\end{tabular}
\caption{For each integer $n\in[0,26]$,  are displayed the $3/2$-expansion $\langle n \rangle_{3/2}$ of $n$,  the value of the sum-of-digits $\mathbf{s}(n)$ in base $3/2$,  and the corresponding value of $\mathbf{s}(n)$ modulo $2$,  i.e.,  the symbol $t_n$ of the Thue--Morse word $\mathbf{t}_{3/2}$ in base $3/2$.}
\label{tab:base32expansions}
\end{table}

A convenient way to visualize $3/2$-expansions in this numeration system is to construct the associated tree (see~\cref{fig:tree32}).
In this tree, under a breadth-first traversal, the vertices have degree $2$ or $1$ alternately. For vertices of degree $2$, the outgoing edges are labeled $0$ and $2$, while for vertices of degree 1, the unique outgoing edge is labeled $1$. Thus, the tree is constructed by an elementary and periodic process --- a \emph{periodic rhythm}.  For each integer $n\ge 0$,  the path from the root to the $n$-th visited vertex in breadth-first search represents $n$ in the numeration system: its label is $\langle n \rangle_{3/2}$. To avoid expansions starting with $0$, it is assumed that the root has a (hidden) loop of label $0$.
Note that any two adjacent vertices are the parents of three vertices. This phenomenon explains why $2$-block substitutions with images of length $3$ naturally arise in this context.

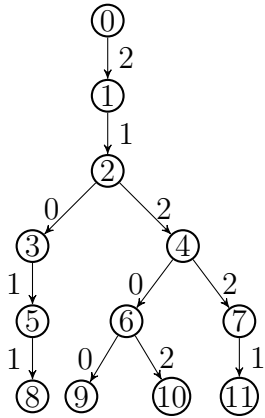
\begin{figure}[h!t]
  \centering
  \tikzset{
  s_bla/.style = {circle,fill=white,thick, draw=black, inner sep=0pt, minimum size=12pt},
  s_red/.style = {circle,black,fill=gray,thick, inner sep=0pt, minimum size=5pt}
}
\begin{tikzpicture}[->,>=stealth',level/.style={sibling distance = 6cm/#1},level distance = 1cm]
  \node [s_bla] {0}
  child {node [s_bla] {1} 
  child {node [s_bla] {2} 
      child {node [s_bla] {3} 
        child {node [s_bla] {5} 
          child {node [s_bla] {8} edge from parent node[left] {$1$} }
          edge from parent node[left] {$1$} }
      edge from parent node[left] {$0$} }
    child {node [s_bla] {4} 
      child {node [s_bla] {6}
        child {node [s_bla] {9} edge from parent node[left] {$0$}}
        child {node [s_bla] {10} edge from parent node[right] {$2$}}
        edge from parent node[left] {$0$}} 
      child {node [s_bla] {7}
        child {node [s_bla] {11} edge from parent node[right] {$1$} }
        edge from parent node[right] {$2$}} 
      edge from parent node[right] {$2$}}
    edge from parent node[right] {$1$}}
   edge from parent node[right] {$2$}}
;
\end{tikzpicture}
  \caption{The first levels of the tree associated with expansions in base $3/2$.}
  \label{fig:tree32}
\end{figure}

\begin{definition}[Our sequence of interest]
The sequence $\mathbf{t}_{3/2}=(t_n)_{n\ge 0}$, whose prefix is given in \eqref{eq:t32}, admits several equivalent descriptions. First, as in the Thue--Morse word recalled in~\cref{ssec:sub},  its $n$-th symbol is the sum-of-digits of $\langle n \rangle_{3/2}$ reduced modulo~$2$; again see~\cref{tab:base32expansions}.  Hence, it can also be generated using a deterministic finite automaton with output; see~\cref{Fig: DFAOt32}.
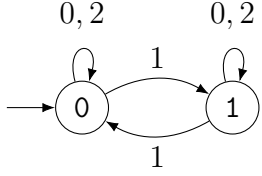
\begin{figure}
\begin{center}
\begin{tikzpicture}
\tikzstyle{every node}=[shape=circle, fill=none, draw=black,minimum size=20pt, inner sep=2pt]
\node(1) at (0,0) {${\tt 0}$};
\node(2) at (2,0) {${\tt 1}$};

\tikzstyle{every node}=[shape=circle, minimum size=5pt, inner sep=2pt]

\draw [-Latex] (-1,0) to node [above] {} (1);

\draw [-Latex] (1) to [loop above] node [above] {$0,2$} (1);
\draw [-Latex] (1) to [bend left] node [above] {$1$} (2);
\draw [-Latex] (2) to [loop above] node [above] {$0,2$} (2);
\draw [-Latex] (2) to [bend left] node [below] {$1$} (1);

\end{tikzpicture}
\end{center}
\caption{A DFAO generating the Thue--Morse sequence $\mathbf{t}_{3/2}$ in base $3/2$.}
\label{Fig: DFAOt32}
\end{figure}
Second, $\mathbf{t}_{3/2}$ is an alternating fixed point of the morphisms \cite[Sec.~3]{RigSti1}
    \begin{equation}\label{eq:alternateTM}
  f_0 \colon \left\{\begin{array}{l}
                      {\tt 0} \mapsto {\tt 00}, \\
                      {\tt 1} \mapsto {\tt 11},\\
                    \end{array}\right.\text{ and }
  f_1\colon \left\{\begin{array}{l}
                      {\tt 0} \mapsto {\tt 1},\\
                      {\tt 1} \mapsto {\tt 0},\\
                    \end{array}\right.
                \end{equation}
                i.e., $\mathbf{t}_{3/2}=f_0(t_0)f_1(t_1)f_0(t_2)f_1(t_3)\cdots={\tt 00111011111011011}\cdots $. Third, it is the fixed point of a uniform $2$-block substitution (every image of a $2$-block has length~$3$),  namely
\begin{equation}\label{eq:tau}
  \tau:\left\{
  \begin{array}{l}
    {\tt 00}\mapsto f_0({\tt 0})f_1({\tt 0})  = {\tt 001},\\
    {\tt 01}\mapsto f_0({\tt 0})f_1({\tt 1})  =  {\tt 000},\\
    {\tt 10}\mapsto f_0({\tt 1})f_1({\tt 0})  =  {\tt 111},\\
    {\tt 11}\mapsto f_0({\tt 1})f_1({\tt 1})  =  {\tt 110}.\\
  \end{array}\right.
\end{equation}
Hence, the sequence $\mathbf{t}_{3/2}=(t_n)_{n\ge 0}$ satisfies the relations
\begin{equation}
  \label{eq:t32relations}
  t_{3n}=t_{3n+1}=t_{2n} \quad\text{ and }\quad 
t_{3n+2}=1-t_{2n+1}
\end{equation}
for $n\ge 0$. One may notice the similarities with~\eqref{eq:tmbase2}.
\end{definition}

In this paper,  we let $\bar{\cdot}$ denote the {\em bit-wise complementation} morphism defined by $\bar{a}=1-a$ for $a\in\cA$. 
Note that the set
\[
  \{\mathbf{x}=(x_n)_{n\ge 0}\in\cA^\mathbb{N} \mid x_{3n}=x_{3n+1}=x_{2n} \text{ and }
  x_{3n+2}=1-x_{2n+1} \text{ for all $n\ge 0$}\}
\]
contains exactly the two sequences $\mathbf{t}_{3/2}$ and $\overline{\mathbf{t}_{3/2}}$ because a sequence in the set is completely determined by its first element.


\subsection{Dekking's variation}\label{ssec:dekk}
Dekking~\cite{DekTM} proposes to use an alternative to the base-$3/2$ numeration system where a natural number $n$ is written instead as
\[
  n=\sum_{i\ge 0}d_i\left(\frac{3}{2}\right)^i
\]
with digits $d_i\in\{0,1,2\}$. Note that, unlike the numeration system considered in \cref{ssec:32TM}, this expansion does not include the normalizing factor $1/2$ as in~\cref{eq: 32expansions}. In this case, the analogue $\mathbf{t}'$ (\seqnum{A357448}) of the Thue--Morse sequence,  starting with
\[
  \mathbf{t}'={\tt 0100101011011010101011011}\cdots,
\]
is the fixed point of the $2$-block substitution
${\tt 00}\mapsto {\tt 010}$, ${\tt 01}\mapsto {\tt 010}$, ${\tt 10}\mapsto {\tt 101}$, and ${\tt 11}\mapsto {\tt 101}$. The sequences $\mathbf{t}'$ and $\mathbf{t}_{3/2}$ are closely related, as shown in the following lemma. 

\begin{lemma}\label{lem:morphic_image}
  Let $\varphi:{\tt 0}\mapsto {\tt 010}$ and ${\tt 1}\mapsto {\tt 101}$. We have
  $\mathbf{t}'=\varphi(\mathbf{t}_{3/2})$. 
\end{lemma}

\begin{proof}
  In the $2$-block substitution generating $\mathbf{t}'$, the image of a block $ab\in\cA^2$ depends only on its first letter $a$,  i.e.,  $ab\mapsto a(1-a)a$ with $a,b\in\{{\tt 0},{\tt 1}\}$. So $\mathbf{t}'$ is also an alternated fixed point of
  \begin{equation}
  g_0 \colon \left\{\begin{array}{l}
                      {\tt 0} \mapsto {\tt 010},\\
                      {\tt 1} \mapsto {\tt 101},\\
                    \end{array}\right. \text{ and }
  g_1\colon \left\{\begin{array}{l}
                      {\tt 0} \mapsto \varepsilon,\\
                      {\tt 1} \mapsto \varepsilon.\\
                    \end{array}\right.
                \end{equation}
                Observe that $\mathbf{t}'$ is also a fixed point of the alternate $3$-block substitution
     \begin{equation*}
  g_0' \colon \left\{\begin{array}{l}
                      {\tt 010} \mapsto g_0({\tt 0})g_1({\tt 1})g_0({\tt 0})={\tt 010}{\tt 010},\\
                      {\tt 101}\mapsto g_0({\tt 1})g_1({\tt 0})g_0({\tt 1})={\tt 101}{\tt 101},\\
                    \end{array}\right.\text{ and }
  g_1'\colon \left\{\begin{array}{l}
                      {\tt 010} \mapsto g_1({\tt 0})g_0({\tt 1})g_1({\tt 0})={\tt 101},\\
                      {\tt 101} \mapsto g_1({\tt 1})g_0({\tt 0})g_1({\tt 1})= {\tt 010}.\\
                    \end{array}\right.
                \end{equation*}
                If we identify through $\varphi$ blocks ${\tt 010}$ and ${\tt 101}$ with $a$ and $b$ respectively, we have
        \begin{equation*}
  g_0' \colon \left\{\begin{array}{l}
                      a \mapsto aa, \\
                      b \mapsto bb,\\
                    \end{array}\right.\text{ and }
  g_1'\colon \left\{\begin{array}{l}
                      a \mapsto b,\\
                      b \mapsto a.\\
                    \end{array}\right.
                \end{equation*}
We know from \eqref{eq:alternateTM} that $\mathbf{t}_{3/2}$ is an alternated fixed point of $(g_0',g_1')$ over $\{a,b\}$. Hence the conclusion follows: $\mathbf{t}'= \varphi(\mathbf{t}_{3/2})$.  
\end{proof}

As for the Oldenburger--Kolakoski sequence $\mathbf{k}$,  Dekking raises a series of conjectures about $\mathbf{t}'$ \cite{DekTM}.
\begin{enumerate}[label=$(C_{\arabic*})$]
\item It is unknown whether $\mathbf{t}'$ is uniformly recurrent.
  
\item It is open whether the set of factors of $\mathbf{t}'$ is closed under bit-wise complement.

\item It is open whether the set of factors of $\mathbf{t}'$ is closed under reversal.

\item It is conjectured that frequencies of the words $w\in\cA^*$ occurring in $\mathbf{t}'$ exist. It is also conjectured that $w$ and its reversal have the same frequency.
\end{enumerate}
Similar questions can also be asked for $\mathbf{t}_{3/2}$.


\subsection{Organization of the paper and our contributions}
\label{sec: contrib}

The first part of the article is purely combinatorial. In~\cref{sec:fd}, we answer the following conjectures asked by Dekking.
\begin{itemize}
\item We prove with~\cref{cor:unif_rec} that both sequences $\mathbf{t}_{3/2}$ and $\mathbf{t}'$ are uniformly recurrent, answering $(C_1)$ positively.
\item We prove with~\cref{pro:close_bw} that the sets of factors of $\mathbf{t}_{3/2}$  and $\mathbf{t}'$ are closed under bit-wise complement, answering $(C_2)$ positively.
\item We prove with~\cref{pro:close_reversals} that the sets of factors of $\mathbf{t}_{3/2}$  and $\mathbf{t}'$ are closed under reversal, answering $(C_3)$ positively. 
\end{itemize}
The strategy consists in proving that the first difference sequence of $\mathbf{t}_{3/2}$ is a Toeplitz word (this is~\cref{pro:ytop}) and thus uniformly recurrent (using~\cref{pro: Toeplitz words are unif rec}).  Passing to the sequence of differences results in a loss of information about the original sequence, which can only be reconstructed up to complementation. It is therefore crucial to prove stability under the bit-wise complement. Results for $\mathbf{t}'$ are deduced from those on $\mathbf{t}_{3/2}$ thanks to \cref{lem:morphic_image}.

\medskip

The second part of the article  is analytic. In \cref{sec:freq}, we prove our main result with \cref{thm:existence_freqs_iff_mun_exist}: frequencies of {\tt 0} and {\tt 1} exist in $\mathbf{t}_{3/2}$ and equal $1/2$. In particular, this also answers $(C_4)$ for symbols. The proof establishes the existence and exact value of ``filtered'' frequencies (along positions congruent to $k$ modulo $2^n$) in $\mathbf{t}_{3/2}$ by combining desubstitution with harmonic analysis on the $2$-adic integers.
The argument is divided into several steps as follows.

\begin{itemize}
\item First,  we argue by contradiction: assuming a deviation from the expected frequency $2^{-n-1}$, \cref{lem:main_compacity_argument} uses a compactness and diagonal extraction argument (via Bolzano--Weierstrass) to construct limiting densities $\mu_n(c,k)$ along a subsequence and allows us to derive recurrence relations coming from the desubstitution of $\mathbf{t}_{3/2}$.
\item Then,  \cref{pro:recrel}  shows that any family of densities $\mu_n(c,k)$ satisfying periodicity, normalization, and these recurrences must equal $2^{-n-1}$. To prove this rigidity, the problem is lifted to the $2$-adic integers $\mathbb{Z}_2$, where differences  are studied as functions in $L^2(\mathbb{Z}_2)$. Using Pontryagin duality and Fourier expansion over the dyadic rationals, the recurrence is reformulated as a linear operator $\Lop$ acting on Fourier coefficients. A spectral contraction estimate --- namely $\|\zeta_2\|_\infty<1$, proved in \cref{sec:33} --- is obtained through explicit computation of the associated multipliers, implying that repeated application of 
$\Lop$ forces the differences to vanish. This yields uniqueness of the solution and the result.
\end{itemize}

To the best of our knowledge,  this is the first time that $2$-adic harmonic analysis is used to prove frequency existence for a rational-base Thue--Morse-type sequence.

\medskip

We finish the paper with~\cref{sec: conclusion} where we expose several paths of future research.


\section{Combinatorial properties of $\mathbf{t}_{3/2}$ and $\mathbf{t}'$}\label{sec:fd}

In this section,  we establish combinatorial properties of the sequences $\mathbf{t}_{3/2}$ and $\mathbf{t}'$.
In particular,  we show that both are uniformly recurrent,  and that their sets of factors are closed under bit-wise complement and reversal.
The strategy is to analyze the sequence of first differences of $\mathbf{t}_{3/2}$,  which shows a particular structure from which we derive the combinatorial properties of the original sequence.

\subsection{On the sequence of first differences}

Let $\Delta$ be the first difference operator defined by $\Delta((x_n)_{n\ge 0}) = (x_{n+1}-x_n \bmod{2})_{n\ge 0}$ (note that the minus sign can be replaced by a plus sign).
The first difference sequence $\Delta(\mathbf{t}_{3/2})$ of $\mathbf{t}_{3/2}$ starts with
\[
  \Delta(\mathbf{t}_{3/2})={\tt 010011000011011000010001011010}\cdots.
\]
It turns out that $\Delta(\mathbf{t}_{3/2})$ is simpler to analyze, as we now show that it is a Toeplitz word.

\begin{definition}
For a finite word $u$, we let $u^\omega$ denote the infinite word obtained by concatenating infinitely many copies of $u$. 
Fix an alphabet $A$ and let $\que$ be a symbol not belonging to $A$. 
For a word $w\in A(A \cup \{\que\})^*$,  we define a converging sequence $(\mathcal{T}_i(w))_{i\ge 0}$ of infinite words in an iterative way.  We let $\mathcal{T}_0(w) := \que^\omega$ and,  for each $i\ge 0$,  we set $\mathcal{T}_{i+1}(w) := F_w(\mathcal{T}_i(w))$,  where,  for any infinite word $\mathbf{u}\in (A \cup \{\que\})^\mathbb{N}$,  we let $F_w(\mathbf{u})$ denote the word obtained from $\mathbf{u}$  by replacing all occurrences of $\que$ by $w^\omega$.
In particular,  $F_w(\mathbf{u}) = \mathbf{u}$, if $\mathbf{u}$ contains no occurrence of $\que$.
The limit
\[
\mathcal{T}(w) = \lim_{i\to \infty} \mathcal{T}_i(w) \in A^\mathbb{N}
\]
is well-defined (because the first letter of $w$ is not the symbol $\que$) and is referred to as the \emph{Toeplitz word determined by the pattern $w$}.
Let $p=|w|$ and $q=|w|_{\que}$ be the length of $w$ and the number of $\que$'s in $w$, respectively. 
We call $\mathcal{T}(w)$ a \emph{$(p,q)$-Toeplitz word}.
\end{definition}

\begin{example}
The paper-folding word is the Toeplitz word determined by the pattern ${\tt 1}\que{\tt 0}\que$; see~\cite{Allouche-Bacher-1992}. 
\end{example}

We recall some results about Toeplitz words \cite{Cassaigne-Karhumaki-1997}.

\begin{theorem}[{\cite{Cassaigne-Karhumaki-1997}}]
\label{thm: factor complexities of Toeplitz words}
Let $\mathbf{x}$ be a $(p,q)$-Toeplitz word and define $d=\gcd(p,q)$,  $p'=p/d$,  and $q'=q/d$.
The factor complexity $\mathsf{p}_\mathbf{x}$ of $\mathbf{x}$ satisfies $\mathsf{p}_\mathbf{x}(n) = \Theta(n^r)$ with $r = \frac{\log p'}{\log (p'/q')}$,  i.e.,  there exist two positive constants $C_1$ and $C_2$ such that $C_1\, n^r \le \mathsf{p}_\mathbf{x}(n) \le C_2\, n^r$ for all $n\ge 0$.
\end{theorem}

\begin{proposition}[{\cite[Sec.~2]{Cassaigne-Karhumaki-1997}}]
\label{pro: Toeplitz words are unif rec}
Toeplitz words are uniformly recurrent.
\end{proposition}

\begin{lemma}
  The sequence $\Delta(\mathbf{t}_{3/2})$ is $3/2$-automatic.
\end{lemma}

\begin{proof}
Write $\Delta(\mathbf{t}_{3/2})=\mathbf{y}=(y_n)_{n\ge 0}$. 
For $n\ge 0$,  we show that,  if $\langle n\rangle_{3/2}=p0u$ with $u\in\{1,2\}^*$,  then the value of $y_n$ is given by $|u|\pmod{2}$ (this also handles the case where $p$ is empty).   We distinguish two cases. 
 
{\bf Case 1.} In the tree associated with the numeration language (recall~\cref{fig:tree32}),  if at some level of the tree two vertices are adjacent, they represent two consecutive numbers,  say $\langle n\rangle_{3/2}$ and $\langle n+1\rangle_{3/2}$.  As mentioned in the introduction,  the rhythm $(02,1)^\mathbb{N}$ of this tree is periodic: vertices of degree $2$, whose outgoing edges are labeled $0$ and $2$, alternate with vertices of degree $1$,  having an outgoing edge labeled $1$. 
The vertices $\langle n\rangle_{3/2}$ and $\langle n+1\rangle_{3/2}$ have a (last) common ancestor $z$,  the lowest one in the tree,  which has outgoing degree $2$. 
  The path from $z$ to $\langle n\rangle_{3/2}$ first uses the left branch from $z$ with label $0$ and then follows the rightmost edge; let $0u$ be its label. In particular, $u\in\{1,2\}^*$. Similarly, the path from $z$ to $\langle n+1\rangle_{3/2}$ first uses the right branch from $z$ with label $2$ and then always follows the leftmost edge; let $2v$ be its label with $v\in\cA^*$. 
Because of the periodic rhythm, $u_i=1$ (resp.,   $2$) if and only if $v_i=0$ (resp.,   $1$). Thus, $u_i+v_i=1 \pmod{2}$. Finally, if $w$ is the label of the path from the root of the tree to $z$, then $t_n$ (resp.,   $t_{n+1}$) is the sum modulo $2$ of the letters of $wu$ (resp.,   $wv$),  which yields
\[
 t_n+t_{n+1}=2\sum_{i=1}^{|w|} w_i + \sum_{i=1}^{|u|} (u_i+v_i)\equiv |u| \pmod{2},
\]
which proves our claim on $y_n$.

{\bf Case 2.}
There is another situation to take into account: if $\langle n\rangle_{3/2}$ is the rightmost vertex of a level, then $\langle n+1\rangle_{3/2}$ is the leftmost vertex of the next level. Once again, the periodic rhythm imposes the following condition:
if the rightmost edge at level $j$ is labeled $1$ (resp.,   $2$), then the leftmost edge at level $j+1$ is labeled $0$ (resp.,   $1$), for all $j$. Thus, if $u=u_1\cdots u_k$ is the path from the root to $\langle n\rangle_{3/2}$, then $2(u_1-1)\cdots (u_k-1)$ is the path to $\langle n+1\rangle_{3/2}$. In that case,
\[
  t_n+t_{n+1}=2+\sum_{i=1}^{|u|} (2u_i-1)\equiv |u| \pmod{2},
\]
which again proves our claim on $y_n$.

The previous property satisfied by the letters $y_n$ of $\mathbf{y}$ can be translated into the deterministic finite automaton with output (DFAO) depicted in~\cref{Fig: DFAOy},  showing that $\mathbf{y}$ is $3/2$-automatic.
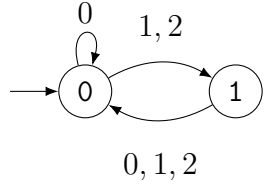
\begin{figure}[htb]
\begin{center}
\begin{tikzpicture}
\tikzstyle{every node}=[shape=circle, fill=none, draw=black,minimum size=20pt, inner sep=2pt]
\node(1) at (0,0) {${\tt 0}$};
\node(2) at (2,0) {${\tt 1}$};

\tikzstyle{every node}=[shape=circle, minimum size=5pt, inner sep=2pt]

\draw [-Latex] (-1,0) to node [above] {} (1);

\draw [-Latex] (1) to [loop above] node [above] {$0$} (1);
\draw [-Latex] (1) to [bend left] node [above] {$1,2$} (2);
\draw [-Latex] (2) to [bend left] node [below] {$0,1,2$} (1);

\end{tikzpicture}
\end{center}
\caption{A DFAO generating the sequence of first differences $\Delta(\mathbf{t}_{3/2})$ of the Thue--Morse sequence $\mathbf{t}_{3/2}$ in base $3/2$.}
\label{Fig: DFAOy}
\end{figure}
\end{proof}

\begin{remark}
  This situation parallels that of the classical Thue--Morse word $\mathbf{t}=(t_n)_{n\ge 0}$ and the period doubling word $\mathbf{p}=(p_n)_{n\ge 0}$; except that the role of ${\tt 0}$ and ${\tt 1}$ are exchanged with $p_n=1+t_n+t_{n+1}\pmod{2}$. One can proceed with the same proof as above,  with the benefit that the tree associated with the base-$2$ system being a full binary tree leads to simpler arguments. Hence $\Delta(\mathbf{t}_{3/2})$ can be seen as an analogue of the period doubling word for base~$3/2$.
\end{remark}

Thanks to \cite[Prop.~16]{RigSti1}, the sequence $\Delta(\mathbf{t}_{3/2}$ is an alternated fixed point of $(f_0,f_1)$ with 
\begin{equation}
  \label{eq:mortoy}
  f_0 \colon \left\{\begin{array}{l}
                      {\tt 0} \mapsto{\tt  01},\\
                      {\tt 1}\mapsto{\tt  00},\\
                    \end{array}\right.\quad \text{ and }
  f_1\colon \left\{\begin{array}{l}
                      {\tt 0} \mapsto{\tt  1},\\
                      {\tt 1}\mapsto{\tt  0}.\\
                    \end{array}\right.
                \end{equation}
We may notice that $f_0$ is the morphism generating the usual period doubling word.

\begin{proposition}\label{pro:ytop}
The sequence $\Delta(\mathbf{t}_{3/2})$ is the $(9,4)$-Toeplitz word $\mathcal{T}({\tt 01}\que{\tt 0}\que{\tt 10}\que\que)$.   
\end{proposition}
\begin{proof}
Write $\Delta(\mathbf{t}_{3/2})=\mathbf{y}=(y_n)_{n\ge 0}$. 
On the one hand,  let $w={\tt 01}\que{\tt 0}\que{\tt 10}\que\que$ and $\mathcal{T}_0(w)=\que^\omega$. 
As prescribed,  we iteratively apply the replacement of the $\que$-symbols to get the first two iterations
\begin{align*}
  \mathcal{T}_1(w)&={\tt 01}\que{\tt 0}\que{\tt 10}\que\que{\tt 01}\que{\tt 0}\que{\tt 10}\que\que{\tt 01}\que{\tt 0}\que{\tt 10}\que\que{\tt 01}\que{\tt 0}\que{\tt 10}\que\que\cdots, \\
  \mathcal{T}_2(w)&={\tt 0100110}\que{\tt 001}\que{\tt 01100}\que{\tt 01}\que{\tt 00101}\que{\tt 0100}\que{\tt 1010}\cdots.
\end{align*}
On the other hand,  we know that
  \[
    \mathbf{y}=(y_n)_{n\ge 0}=f_0(y_0) f_1(y_1) f_0(y_2) f_1(y_3) f_0(y_4) f_1(y_5)\cdots,
  \]
  so the word $\mathbf{y}$ is made of blocks of length $3$ of the form $f_0(y_{2n})f_1(y_{2n+1})$ with $n\ge 0$. From the definition of the two morphisms in \eqref{eq:mortoy}, a direct inspection yields
  \[
    y_{3n}={\tt 0},\quad y_{3n+1}=\overline{y_{2n}},\quad y_{3n+2}=\overline{y_{2n+1}}
  \]
  for all $n\ge 0$ (where we recall that $\bar{\cdot}$ is bit-wise complementation morphism).
  Since the given pattern $w$ has length~$9$, consider blocks of length $9$ in $\mathbf{y}$. The above relations give,  for all $m\ge 0$,
  \begin{align*}
    y_{9m}&=y_{9m+3}=y_{9m+6}={\tt 0},\\
    y_{9m+1}&=y_{3\cdot 3m+1}=\overline{y_{3\cdot 2m}}={\tt 1},\\
    y_{9m+5}&=y_{3 \cdot(3m+1)+2}=\overline{y_{2(3m+1)+1}}=\overline{y_{3 \cdot (2m+1)}}={\tt 1},
  \end{align*}
  which correspond to the ``non-holes'' positions (i.e., the non-$\que$ symbols) in the pattern $w$.  Now,  for all $m\ge 0$,  we also have
  \begin{align*}
    y_{9m+2}&=y_{3\cdot 3m+2}=\overline{y_{3 \cdot2m+1}}=y_{4m},\\
    y_{9m+4}&=y_{3\cdot(3m+1)+1}=\overline{y_{3\cdot2m+2}}=y_{4m+1},\\
    y_{9m+7}&=y_{3\cdot(3m+2)+1}=\overline{y_{3\cdot(2m+1)+1}}=y_{4m+2},\\
    y_{9m+8}&=y_{3\cdot(3m+2)+2}=\overline{y_{3\cdot(2m+1)+2}}=y_{4m+3}.
  \end{align*}
This exactly corresponds to the Toeplitz construction: the holes in the pattern $w={\tt 01}\que{\tt 0}\que{\tt 10}\que\que$ have positions congruent to $2,4,7,8\pmod{9}$ and they are replaced by symbols appearing earlier in the sequence, in consecutive classes modulo $4$.
This shows that $\mathbf{y}=\mathcal{T}({\tt 01}\que{\tt 0}\que{\tt 10}\que\que)$,  as desired.
\end{proof}

Every regular\footnote{A Toeplitz word $\mathcal{T}(w)$ is \emph{regular} if the density of $\que$-symbols in $\mathcal{T}_i(w)$ tends to $0$ as $i$ goes to infinity.  It is the case with our $(9,4)$-periodic Toeplitz word $\mathcal{T}({\tt 01}\que{\tt 0}\que{\tt 10}\que\que)$ where $(4/9)^i$ indeed tends to $0$ as $i$ goes to infinity.} Toeplitz system is \emph{strictly ergodic} (i.e.,  minimal and uniquely ergodic) \cite{JacobsKeane1969}. Hence, the frequency of every factor exists.
The letter frequencies in $\Delta(\mathbf{t}_{3/2})$ thus exist by~\cref{pro:ytop} and can be computed explicitly. For instance,  for all $i\ge 0$,  let $P_i$ be the word of length $9^i$ such that $\mathcal{T}_i(w)=P_i^\omega$. From the Toeplitz generating process, we get
\[
  \begin{pmatrix}
    |P_{i+1}|_{\tt 0}\\
    |P_{i+1}|_{\tt 1}\\
    |P_{i+1}|_{\que}\\
  \end{pmatrix}=
  \begin{pmatrix}
    9&0&3\\
    0&9&2\\
    0&0&4\\
  \end{pmatrix}\begin{pmatrix}
    |P_{i}|_{\tt 0}\\
    |P_{i}|_{\tt 1}\\
    |P_{i}|_{\que}\\
  \end{pmatrix}
\]
since $\que$-symbols are replaced by three ${\tt 0}$'s,  two ${\tt 1}$'s and four $\que$'s.
Hence, for all $n\ge 1$, 
\[
  \frac{1}{9^n}
  \begin{pmatrix}
    |P_{n}|_{\tt 0}\\
    |P_{n}|_{\tt 1}\\
    |P_{n}|_{\que}\\
  \end{pmatrix}= \frac{1}{9^n}\begin{pmatrix}
    9&0&3\\
    0&9&2\\
    0&0&4\\
  \end{pmatrix}^n
  \begin{pmatrix}
    0\\0\\1
  \end{pmatrix},
\]
from which it follows that $\freq_{\Delta(\mathbf{t}_{3/2})}({\tt 0})=3/5$ and $\freq_{\Delta(\mathbf{t}_{3/2})}({\tt 1})=2/5$.


\subsection{From the sequence $\Delta(\mathbf{t}_{3/2})$ back to the original sequence $\mathbf{t}_{3/2}$}

Every factor $d_1\cdots d_\ell$ in $\Delta(\mathbf{t}_{3/2})$ corresponds to one of the two complementary factors
\begin{align*}
              & {\tt 0}({\tt 0}+d_1\bmod{2})({\tt 0}+d_1+d_2\bmod{2})\cdots ({\tt 0}+d_1+\cdots+d_\ell\bmod{2})  \\
  \text{ or, } & {\tt 1}({\tt 1}+d_1\bmod{2})({\tt 1}+d_1+d_2\bmod{2})\cdots ({\tt 1}+d_1+\cdots+d_\ell\bmod{2})
\end{align*}
in $\mathbf{t}_{3/2}$.  
At least one of the two occurs in $\mathbf{t}_{3/2}$.  Hence we obtain
\[
  \mathsf{p}_{\Delta(\mathbf{t}_{3/2})}(n) \le \mathsf{p}_{\mathbf{t}_{3/2}}(n+1)\le 2 \mathsf{p}_{\Delta(\mathbf{t}_{3/2})}(n)
\]
for all $n\ge 0$.
However, this is not enough to conclude about frequencies in $\mathbf{t}_{3/2}$. If ${\tt 0}u$ and ${\tt 1}\bar{u}$ are such that $\Delta({\tt 0}u)=\Delta({\tt 1}\bar{u})=v$ (recall that the bit-wise complementation morphism $\bar{\cdot}$ was defined at the end of~\cref{ssec:32TM}),  then we only have information about the combined frequencies
\[
  \lim_{N\to\infty} \frac{|\mathbf{t}_{3/2}[0,N)|_{{\tt 0}u}+ |\mathbf{t}_{3/2}[0,N)|_{{\tt 1}\bar{u}}}{N}=\freq_{\Delta(\mathbf{t}_{3/2})}(v).
\]
Nonetheless,  the structure of $\Delta(\mathbf{t}_{3/2})$ allows us to deduce various combinatorial properties of $\mathbf{t}_{3/2}$,  as we show next.

\paragraph{Uniform recurrence.} First,  we prove with~\cref{cor:unif_rec} that both words $\mathbf{t}_{3/2}$ and $\mathbf{t}'$ are uniformly recurrent as a particular case of~\cref{pro: Toeplitz words are unif rec} and the next result.

\begin{proposition}\label{pro:fdifur}
The binary word $\mathbf{x}$ is uniformly recurrent if and only if the first difference sequence $\Delta(\mathbf{x})$ of $\mathbf{x}$ is uniformly recurrent.
\end{proposition}
\begin{proof}
For the sake of readability,  set $\mathbf{y}=\Delta(\mathbf{x})$.

First assume that $\mathbf{x}$ is uniformly recurrent.
Take a factor $v$ of $\mathbf{y}$.
Then there exists a factor $u$ of $\mathbf{x}$ such that $\Delta(u)=v$.
By assumption,  $u$ occurs in $\mathbf{x}$ with bounded gaps,  so does $v$ in $\mathbf{y}$,  which is enough.

Conversely, assume that $\mathbf{y}$ is uniformly recurrent.
Take a factor $u$ of $\mathbf{x}$.
We consider two cases depending on whether the complement of $u$ is a factor of $\mathbf{x}$ or not.

\textbf{Case 1.} If $\bar{u}$ is not a factor of $\mathbf{x}$,  then each occurrence of $\Delta(u)$ in $\mathbf{y}$ corresponds to an occurrence of $u$ in $\mathbf{x}$. 
Since $\Delta(u)$ occurs with bounded gaps,  we may conclude.

\textbf{Case 2.} Next suppose that $\bar{u}$ is a factor of $\mathbf{x}$ and consider a factor $w$ containing both $u$ and $\bar{u}$ in $\mathbf{x}$.
Now $\Delta(w)$ occurs with bounded gaps in $\mathbf{y}$,  then $\{w,\bar{w}\}$ occurs with bounded gaps in $\mathbf{x}$.
Since both $w$ and $\bar{w}$ contain $u$ as a factor (because $\bar{\bar{u}}=u$),  then $u$ occurs with bounded gaps in $\mathbf{x}$.
\end{proof}

\begin{corollary}\label{cor:unif_rec}
  The sequences $\mathbf{t}_{3/2}$ and $\mathbf{t}'$ are uniformly recurrent. 
\end{corollary}

\begin{proof}
  By~\cref{pro:ytop}, $\Delta(\mathbf{t}_{3/2})$ is a Toeplitz word. 
Since Toeplitz words are uniformly recurrent (see~\cref{pro: Toeplitz words are unif rec}),  so is $\mathbf{t}_{3/2}$ by the previous proposition.
Uniform recurrence of $\mathbf{t}'$ follows from~\cref{lem:morphic_image}.
\end{proof}

\paragraph{Occurences at odd and even positions.} 
The next result tracks the positions of factors in $\mathbf{t}_{3/2}$.
It is crucial in the proof of our closure results,  namely~\cref{pro:close_bw,pro:close_reversals}.
The main technical step is to show that for every $n \ge 0$ and every residue class modulo $2^n$, 
both letters in $\{{\tt 0}, {\tt 1}\}$ occur in $\mathbf{t}_{3/2}$ at positions in that class.
We observe that this also follows from \Cref{thm:existence_freqs_iff_mun_exist}, which establishes the stronger fact 
that the occurrences of each letter in $\mathbf{t}_{3/2}$ is equidistributed among the $2^n$ residue classes.

\begin{proposition}
    \label{lem:parity_recurrence}
    If a word occurs in $\mathbf{t}_{3/2}$, then it occurs at both even and odd positions in $\mathbf{t}_{3/2}$.
\end{proposition}
\begin{proof}
Write $\mathbf{t}_{3/2}=(t_n)_{n\ge 0}$ and recall that $\mathbf{t}_{3/2}[0,j)$ is the prefix of $\mathbf{t}_{3/2}$ of length $j$. 
In particular, for $j=0$, this is the empty prefix. 
From the morphisms $f_0$ and $f_1$ of~\eqref{eq:alternateTM} and the $2$-block substitution~$\tau$ of~\eqref{eq:tau}, we define the map $\Phi$ by
\begin{align*}
\Phi(a)&= f_0(a), \quad a\in\{{\tt 0},{\tt 1}\},\\
\Phi(t_0\cdots t_{2j-1})&=\tau(t_0 t_1) \cdots \tau(t_{2j-2}t_{2j-1}),\quad j\ge 1,\\
\Phi(t_0\cdots t_{2j})&=\tau(t_0 t_1) \cdots \tau(t_{2j-2}t_{2j-1}) f_0(t_{2j}),\quad j\ge 1.
\end{align*}
  
\begin{figure}[h!t]
\centering
\begin{tikzpicture}[>=latex]
  \draw (0,0) rectangle (4,.5);
  \draw (0,0) rectangle (.5,.5);
  \draw (0,0) rectangle (1,.5);
  \draw[fill=gray!20] (1,0) rectangle (1.5,.5);
  \node at (5,0.2) {$t_{n_0(j)}=t_j$};
  
  \draw [decorate, decoration={brace, amplitude=6pt}]
  (0,.6) -- (4,.6)
  node[midway, above=6pt] {$j$};

  \draw (0,-1) rectangle (6.5,-1.5);
  \draw (0,-1) rectangle (1,-1.5);
  \draw (0,-1) rectangle (1.5,-1.5);
  \draw[fill=gray!20] (1.5,-1) rectangle (2.5,-1.5);
  \node at (7.1,-1.3) {$t_{n_1(j)}$};
  
  \draw (0,-2.5) rectangle (9.5,-3);
  \draw (0,-2.5) rectangle (1.5,-3);
  \draw (0,-2.5) rectangle (2.5,-3);
  \draw[fill=gray!20] (2.5,-2.5) rectangle (4,-3);
  \node at (10.1,-2.8) {$t_{n_2(j)}$};
  
  \draw (0,-4) rectangle (14,-4.5);
  \draw (0,-4) rectangle (2.5,-4.5);
  \draw (0,-4) rectangle (4,-4.5);
  \draw[fill=gray!20] (4,-4) rectangle (6,-4.5);
  \node at (14.6,-4.3) {$t_{n_3(j)}$};

  \coordinate (f1) at (.25,0.1);
  \coordinate (g1) at (.5,-1);
  \draw[->] (f1) to (g1);
  \node[left=3pt] at (.6,-.5) {${\small f_0}$};
  \coordinate (f1) at (.75,0.1);
  \coordinate (g1) at (1.25,-1);
  \draw[->] (f1) to (g1);
  \node[left=3pt] at (1.2,-.5) {${\small f_1}$};
  \coordinate (f1) at (1.25,0.1);
  \coordinate (g1) at (2,-1);
  \draw[->] (f1) to (g1);
  \node[left=3pt] at (1.8,-.5) {${\small f_0}$};
\begin{scope}[yshift=-1.5cm]
  \coordinate (f1) at (.25,0.1);
  \coordinate (g1) at (.5,-1);
  \draw[->] (f1) to (g1);
  \node[left=3pt] at (.6,-.5) {${\small f_0}$};
  \coordinate (f1) at (.75,0.1);
  \coordinate (g1) at (1.25,-1);
  \draw[->] (f1) to (g1);
  \node[left=3pt] at (1.2,-.5) {${\small f_1}$};
  \coordinate (f1) at (1.25,0.1);
  \coordinate (g1) at (2,-1);
  \draw[->] (f1) to (g1);
  \node[left=3pt] at (1.8,-.5) {${\small f_0}$};
  \begin{scope}[xshift=1cm]
  \coordinate (f1) at (.75,0.1);
  \coordinate (g1) at (1.75,-1);
  \draw[->] (f1) to (g1);
  \node[left=3pt] at (1.4,-.5) {${\small f_1}$};
  \coordinate (f1) at (1.25,0.1);
  \coordinate (g1) at (2.5,-1);
  \draw[->] (f1) to (g1);
  \node[left=3pt] at (2.,-.5) {${\small f_0}$};
  \end{scope}
\end{scope}
\begin{scope}[yshift=-3cm]
  \coordinate (f1) at (.25,0.1);
  \coordinate (g1) at (.5,-1);
  \draw[->] (f1) to (g1);
  \node[left=3pt] at (.6,-.5) {${\small f_0}$};
  \coordinate (f1) at (.75,0.1);
  \coordinate (g1) at (1.25,-1);
  \draw[->] (f1) to (g1);
  \node[left=3pt] at (1.2,-.5) {${\small f_1}$};
  \coordinate (f1) at (1.25,0.1);
  \coordinate (g1) at (2,-1);
  \draw[->] (f1) to (g1);
  \node[left=3pt] at (1.8,-.5) {${\small f_0}$};
  \begin{scope}[xshift=1cm]
  \coordinate (f1) at (.75,0.1);
  \coordinate (g1) at (1.75,-1);
  \draw[->] (f1) to (g1);
  \node[left=3pt] at (1.4,-.5) {${\small f_1}$};
  \coordinate (f1) at (1.25,0.1);
  \coordinate (g1) at (2.5,-1);
  \draw[->] (f1) to (g1);
  \node[left=3pt] at (2.,-.5) {${\small f_0}$};
   \coordinate (f1) at (1.75,0.1);
  \coordinate (g1) at (3.25,-1);
  \draw[->] (f1) to (g1);
  \node[right=3pt] at (1.8,-.5) {${\small f_1}$};
  \end{scope}
  \end{scope}

  \coordinate (b1) at (-.25,0);
  \coordinate (c1) at (-.25,-1);
  \draw[->] (b1) to (c1);
  \node[left=3pt] at (-.4,-.5) {$\Phi$};
    
  \coordinate (b2) at (-.25,-1.5);
  \coordinate (c2) at (-.25,-2.5);
  \draw[->] (b2) to (c2);
  \node[left=3pt] at (-.4,-2) {$\Phi$};
  
  \coordinate (b3) at (-.25,-3);
  \coordinate (c3) at (-.25,-4);
  \draw[->] (b3) to (c3);
  \node[left=3pt] at (-.4,-3.5) {$\Phi$};

  \draw [decorate, decoration={brace, mirror, amplitude=6pt}]
  (0,-4.6) -- (14,-4.6)
  node[midway, below=6pt] {$|\Phi^3 (\mathbf{t}_{3/2}[0,j))|=n_3(j)$};
\end{tikzpicture}
\caption{Using the fixed point structure,  we iterate the map $\Phi$ three times on successive prefixes of $\mathbf{t}_{3/2}$.}\label{fig:fixp}
\end{figure}
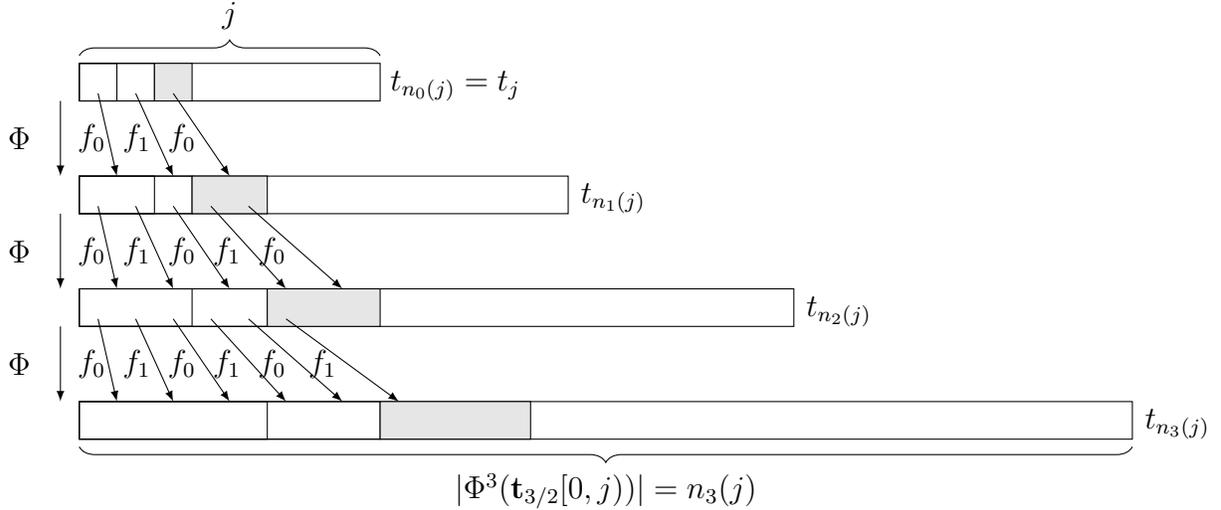

We now track how the parity of positions evolves under iterations of $\Phi$. 
For all $k,j\ge 0$, let
\[
  n_k(j) :=|\Phi^k (\mathbf{t}_{3/2}[0,j))|.
\]
In particular, $n_0(j)=j$ for all $j\ge 0$. 
These quantities allow us to determine which of the two morphisms $f_0$ or $f_1$ is applied at a given position.

To give some intuition, let us track the successive images of the letter $t_2$ when applying $\Phi$, symbolized by the gray rectangles in~\cref{fig:fixp}.
\begin{itemize}
    \item Since $n_0(2)$ is even, we apply $f_0$ and the image of $t_2$ is $f_0(t_2)$ of length~$2$. The position where this image occurs is given by $n_1(2)=3$,  so  $f_0(t_2)=t_3t_4$.
    \item If we apply $\Phi$ again, since $n_1(2)$ is odd, we first apply $f_1$, then $f_0$,  so we consider the word $f_1(t_3)f_0(t_4)$.  The corresponding image is $f_1(t_3)f_0(t_4)=t_5 \cdot  t_6t_7$ because $n_2(2)=5$. 
    \item Again,  since $n_2(2)$ is odd,  applying $\Phi$ leads to  $f_1(t_5)f_0(t_6)f_1(t_7)=t_8 \cdot  t_9t_{10} \cdot t_{11}$ because $n_3(2)=8$.
    \item On the next iteration, since $n_3(2)$ is even, the image is $f_0(t_8) f_1(t_9) f_0(t_{10}) f_1(t_{11})$ and so on and so forth.
\end{itemize}
Observe that in the tree in \cref{fig:tree32}, the sequence $n_k(j)$ can be recovered as follows: we consider the leftmost path in the (infinite) subtree rooted at $j$. 
In this subtree, the leftmost vertex at level $k$ is $n_k(j)$. We note that $n_k(j)$ does not depend on the particular substitution \eqref{eq:tau} but only on the structure of the numeration tree.

Since we apply $f_0$ and $f_1$ alternatively, we observe that if $n_k(j)$ is even, then
\[
  t_{n_{k+1}(j)} t_{n_{k+1}(j)+1}=f_0(t_{n_k(j)}),
\]
and if otherwise $n_k(j)$ is odd,  then
\[
  t_{n_{k+1}(j)}=f_1(t_{n_k(j)}).
\]
Note that arbitrarily long prefixes of $\mathbf{t}_{3/2}$ are obtained by iterating $\Phi$ starting from the symbol~${\tt 0}$.
Since these prefixes start at position $0$, the first morphism applied at each iteration is always $f_0$.
Now, if $\Phi^\ell({\tt 0})$ also occurs at an odd position of $\mathbf{t}_{3/2}$, then every factor of $\Phi^\ell({\tt 0})$ occurs at both even and odd positions in $\mathbf{t}_{3/2}$.
To see this, observe that if a factor $u$ occurs at position $i$ in $\Phi^\ell({\tt 0})$ and if $\Phi^\ell({\tt 0})$ occurs at position $j$ with $j$ odd in $\mathbf{t}_{3/2}$, then $u$ occurs in $\mathbf{t}_{3/2}$ at positions $i$ and $j+i$.
Since $j$ is odd, these two occurrences have opposite parities.
This reduces the problem to proving that $\Phi^\ell({\tt 0})$ occurs at an odd position of $\mathbf{t}_{3/2}$ for all $\ell \ge 0$.
Equivalently, by the discussion above, we must find $j \in \N$ such that $n_0(j), \dots, n_\ell(j)$ are even, $n_{\ell+1}(j)$ is odd, and $t_j = {\tt 0}$.
Indeed, it is the parity that determines which morphism is applied first, so this $j$ would produce $\Phi^\ell(t_j)$ at position $n_{\ell+1}(j)$ of $\mathbf{t}_{3/2}$, whence the result since $t_j = {\tt 0}$.
\medskip

Let $\ell \ge 1$ be an arbitrary integer.
First, we note that if $j \in \N$, then $n_k(2^{\ell} j) = 3^k \cdot 2^{\ell-k} \cdot j$ is even for $k < \ell$, while $n_\ell(2^\ell j) = 3^\ell \cdot j$ has the same parity as $j$.
Therefore, it suffices to find $j \in \N$ such that $t_{2^\ell j} = {\rm 0}$ and $j$ is odd.

Since $3^{\ell+1}$ is invertible modulo $2^{\ell+2}$, there exists $i \in \N$ such that $3^{\ell+1} i = 2^{\ell+1} - 1 \pmod{2^{\ell+2}}$.
We note that this congruence ensures, in particular, that $i$ is odd.
Let $j_0 \in \N$ be the integer satisfying $3^{\ell+1} i = 2^{\ell+1} - 1 + 2^{\ell+2}j_0$.
Define $j = 2j_0+1$, which is odd, and note that $3^{\ell+1} i + 1 = 2^{\ell+1} j$.
We now show that $t_{2^\ell j} \neq t_{2^\ell i}$.
Once this is established, $\mathbf{t}_{3/2}$ contains an occurrence of ${\tt 0}$ at either position $2^\ell i$ or $2^\ell j$.
Since both $i$ and $j$ are odd, it follows that $\Phi^\ell({\tt 0})$ occurs at an odd position of $\mathbf{t}_{3/2}$, completing the argument.

Let $c = t_{2^\ell i} \in \{{\tt 0},{\tt 1}\}$. 
Note that since $2^\ell i$ is even and since $f_0(c)=cc$, we have 
\[  \Phi(\mathbf{t}_{3/2}[0,2^\ell i]) = \Phi(\mathbf{t}_{3/2}[0,2^\ell i)\, c) = 
    \mathbf{t}_{3/2}[0,3\cdot 2^{\ell-1}i) f_0(c) = 
        \mathbf{t}_{3/2}[0,3\cdot 2^{\ell-1}i)\, cc. \]
So $c$ appears in $\mathbf{t}_{3/2}$ in position $3\cdot 2^{\ell-1} i$.
If $\ell \ge 2$, then $3\cdot 2^{\ell-1} i$ is even and we can iterate our argument: $c$ also occurs in $\mathbf{t}_{3/2}$ in position $3^2\cdot 2^{\ell-2}i$. 
Continuing like this, we get that $c$ occurs in $\mathbf{t}_{3/2}$ in position $3^\ell i$, which is odd since $i$ is odd.
Therefore, if we apply $\Phi$ one more time, since the parity changes, we then get
\[  \Phi(\mathbf{t}_{3/2}[0,3^\ell i]) = \Phi(\mathbf{t}_{3/2}[0,3^\ell i)\, c) = 
    \Phi(\mathbf{t}_{3/2}[0,3^\ell i)) f_1(c) = 
    \Phi(\mathbf{t}_{3/2}[0,3^\ell i))\, \bar{c},\]
so the bit-wise complement of $c$ also occurs in $\mathbf{t}_{3/2}$.
We now focus on the position of this occurrence.
The length of $\Phi(\mathbf{t}_{3/2}[0,3^\ell i))$ can be computed using the fact that $3^\ell i - 1$ is even since in that case we have
\[  |\Phi(\mathbf{t}_{3/2}[0,3^\ell i))| = |\Phi(\mathbf{t}_{3/2}[0,3^\ell i-1))|
   + |f_0(t_{3^\ell i-1})| =
        \frac{3}{2}(3^\ell i - 1) + 2 = 2^\ell j,  \]
where in the last step we use that $3^{\ell+1}i + 1 = 2^{\ell+1}j$ by definition of $j$.
We conclude that $\bar{c}$ occurs in $\mathbf{t}_{3/2}$ in position $2^\ell j$ and that $c$ occurs in $\mathbf{t}_{3/2}$ in position $2^\ell i$,  which is enough.
\end{proof}


We present an alternative proof of \cref{cor:unif_rec}. 
It does not rely on the difference word and appears to be more general.

\begin{proof}[Second proof of \cref{cor:unif_rec}]
We begin with some definitions.
Let $\mathbf{s}_{3/2}$ be the fixed point of $\tau$ from~\cref{eq:tau} starting from the letter ${\tt 1}$,  i.e.,  
\[ 
\mathbf{s}_{3/2} =
\lim_{n\to\infty} \tau^n({\tt 1}).
\]
For an infinite word $\mathbf{z} \in \{{\tt0},{\tt1}\}^\N$, we denote by $\cL(\mathbf{z})$ the set of its finite factors.

A classical result in topological dynamics (every system contains a minimal subsystem; see Auslander's book \cite{auslander}) ensures the existence of a uniformly recurrent infinite word $\mathbf{x} \in \{{\tt0},{\tt1}\}^\N$ such that
$\cL(\mathbf{x}) \subseteq \cL(\mathbf{t}_{3/2})$.
Consider the quantities $n_k(j)$, $j,k \in \N$, from the proof of \Cref{lem:parity_recurrence}, and define
\[
N_k(j) := (n_0(j) \bmod 2, n_1(j) \bmod 2, \dots, n_k(j) \bmod 2).
\]

The proof consists of three steps.
First, we show that for each $k \ge 0$, the sequence $(N_k(j))_{j\ge 0}$ is periodic with least period $2^{k+1}$.
Second, we use this to prove that $\cL(\mathbf{x})$ contains either $\cL(\mathbf{t}_{3/2})$ or $\cL(\mathbf{s}_{3/2})$.
Finally, we deduce the uniform recurrence of both $\mathbf{t}_{3/2}$ and $\mathbf{s}_{3/2}$.

\textbf{Step 1.} We begin\footnote{We present the argument for the sake of completeness, it appears in \cite[Lem.~4.14]{Marsault}.} with the periodicity of $N_k(j)$.
For any $j \in \N$ and $\ell \le k$, we have
$n_\ell(2^{k+1} j) = 3^\ell\cdot 2^{k+1-\ell}\cdot j \equiv 0 \pmod 2$.
Thus $N_k(2^{k+1} j) = N_k(2^{k+1} j')$ for all $j,j' \in \N$.
Now let $i \in \N$ be arbitrary and write $i = 2^{k+1} j + j_0$ with $0 \le j_0 < 2^{k+1}$.
Since
$n_\ell(2^{k+1}p + 2^{k+1}p') = n_\ell(2^{k+1}p) + n_\ell(2^{k+1}p')$
for all $p,p' \in \N$, we obtain
$n_\ell(i) = n_\ell(i') \pmod 2$
whenever $i \equiv i' \pmod{2^{k+1}}$.
This proves that $(N_k(j))_{j\ge 0}$ is periodic with period $2^{k+1}$.

Let $p$ be the least period of $(N_k(j))_{j\ge 0}$.
Then, $p$ divides $2^{k+1}$.
Now, note that for any $j \in \N$ we have $n_{k-1}(2^k j) = 3^k\cdot j \equiv j \pmod 2$.
Hence, $(N_k(2^k j))_{j\ge 0}$ is not a constant sequence.
Therefore, if $p \neq 2^{k+1}$ then $p$ divides $2^k$, implying that $(N_k(2^k j))_{j\ge 0}$ is constant, which is a contradiction.

\textbf{Step 2.}
 Consider the maps $f_0$ and $f_1$ defined in~\eqref{eq:alternateTM}.
    For $i \in \{0,1\}$, we define maps $\Phi_i$ by
    \begin{equation}\label{eq:phii}
        \Phi_i(a_0 a_1 \cdots a_{k-1}) = f_{i\bmod{2}}(a_0) f_{{i+1}\bmod{2}}(a_1) \cdots f_{{i+k-1}\bmod{2}}(a_{k-1}),
    \end{equation}
    for all non-empty words $a_0 a_1 \cdots a_{k-1} \in \cA^*$. 
Next, we show that $\cL(\mathbf{x})$ contains $\cL(\mathbf{t}_{3/2})$ or $\cL(\mathbf{s}_{3/2})$.
Fix a positive integer $k$.
By definition of $\mathbf{x}$, one can construct a $k$-fold desubstitution of $\mathbf{x}$: there exist
$\mathbf{y} \in \{{\tt0},{\tt1}\}^\N$ and $i_1,\dots,i_k \in \{0,1\}$ such that 
$\Phi_{i_k} \circ \cdots \circ \Phi_{i_1}(\mathbf{y})$ coincides with $\mathbf{x}$ after deleting a finite prefix.
Now, by the periodicity from the first step, there exists a position $j$ in $\mathbf{y}$ such that reapplying the substitutions
$\Phi_{i_k},\dots,\Phi_{i_1}$ produces the word $\Phi_0^k(\mathbf{y}_j)$.
(Equivalently, this is possible because all subtrees occur in the numeration system.)
Repeating this procedure for all $k \ge 1$, we obtain occurrences in $\mathbf{x}$ of words of the form $\Phi_0^k(a_k)$ for some letters $a_k$.
By the pigeonhole principle, there exists a letter $a$ that appears infinitely often in the sequence $(a_k)_{k\ge 1}$.
Since the words $\Phi_0^k(a)$ are arbitrarily long prefixes of the fixed point of $\tau$ starting with $a$, we conclude that
$\cL(\mathbf{x})$ contains the language of that fixed point, which we know is $\mathbf{t}_{3/2}$ or $\mathbf{s}_{3/2}$
This completes the second step.

\textbf{Step 3.}
Finally, we prove that $\mathbf{t}_{3/2}$ is uniformly recurrent.
Since $\mathbf{t}_{3/2}$ is the bit-wise complement of $\mathbf{s}_{3/2}$, one is uniformly recurrent if and only if the other is.
It therefore suffices to show that the fixed point starting with the letter $a$ in the previous paragraph is uniformly recurrent.
By symmetry, we may assume $a = {\tt 0}$, so that $\cL(\mathbf{x}) \supseteq \cL(\mathbf{t}_{3/2})$.

Let $u$ be a factor of $\mathbf{t}_{3/2}$.
Because $\mathbf{x}$ is uniformly recurrent, there exists $L \ge 1$ such that every factor of $\mathbf{x}$ of length at least $L$ contains an occurrence of $u$.
Since $\cL(\mathbf{x}) \supseteq \cL(\mathbf{t}_{3/2})$, the same holds for factors of $\mathbf{t}_{3/2}$ of length at least $L$.
Hence $\mathbf{t}_{3/2}$ is uniformly recurrent.
\end{proof}


\paragraph{Bit-wise complement.} The second main result of this section is to show that the sets of factors of both words $\mathbf{t}_{3/2}$ and $\mathbf{t}'$ are closed under bit-wise complement,  i.e.,  if a word $u\in\cA^*$ is a factor of one of the words,  so is $\bar{u}$.

\begin{proposition}\label{pro:close_bw}
  The set of factors of $\mathbf{t}_{3/2}$ (and thus of $\mathbf{t}'$) is closed under bit-wise complement. 
  Consequently,  the relationship between the factor complexities of $\mathbf{t}_{3/2}$ and $\Delta(\mathbf{t}_{3/2})$ satisfies $\mathsf{p}_{\mathbf{t}_{3/2}}(n+1)= 2 \mathsf{p}_{\Delta(\mathbf{t}_{3/2})}(n)$ for all $n\ge 0$.
\end{proposition}

\begin{proof}
We continue using the maps $f_0$, $f_1$, and $\Phi$ from the beginning of the proof of \Cref{lem:parity_recurrence}.
A direct computation shows that
$\overline{f_i(a)} = f_i(\bar{a})$ for all $i \in \{0,1\}$ and $a \in \cA$.
It follows by induction that
\begin{equation*}
    \overline{\Phi(u)} = \Phi(\bar{u})
\end{equation*}
for all nonempty words $u \in \{{\tt 0},{\tt 1}\}^*$.
This implies that, for any $\ell \ge 0$, $\Phi^\ell({\tt 1})$ is the bitwise complement of $\Phi^\ell({\tt 0})$.
Since the words $\Phi^\ell({\tt 0})$, $\ell \ge 0$, are arbitrarily long prefixes of $\mathbf{t}_{3/2}$, it follows that $\mathbf{t}_{3/2}$ is closed under bitwise complement if $\Phi^\ell({\tt 1})$ occurs in $\mathbf{t}_{3/2}$ for all $\ell \ge 0$.

To show this,  we proceed by induction on $\ell\ge 0$.
For the base case $\ell = 0$,  $\Phi^0({\tt 1}) = {\rm 1}$ occurs in $\mathbf{t}_{3/2}$.
For the induction step,  assume that $\Phi^\ell({\tt 1})$ occurs in $\mathbf{t}_{3/2}$ for some $\ell \ge 0$.
Then, by \Cref{lem:parity_recurrence}, $\Phi^\ell({\tt 1})$ occurs in an even position of $\mathbf{t}_{3/2}$, and thus $\Phi^{\ell+1}({\tt 1})$ occurs in $\mathbf{t}_{3/2}$,  which finishes the proof.
\end{proof}

Combining the previous result and~\cref{thm: factor complexities of Toeplitz words},  we obtain the following.

\begin{corollary}
The factor complexities $\mathsf{p}_{\mathbf{t}_{3/2}}$ and $\mathsf{p}_{\Delta(\mathbf{t}_{3/2})}$ of $\mathbf{t}_{3/2}$ and $\Delta(\mathbf{t}_{3/2})$ are both in $\Theta(n^r)$ with $r = \frac{\log 3}{\log (3/2)}$.
\end{corollary}


\paragraph{Reversal.} The third main result of this section is to show that the sets of factors of both words $\mathbf{t}_{3/2}$ and $\mathbf{t}'$ are closed under reversal.
The {\em reversal} $u^R$ of a word $u$ is defined by $\varepsilon^R = \varepsilon$, and by $(a_0 a_1 \cdots a_{\ell-1})^R = a_{\ell-1} a_{\ell-2} \cdots a_0$ where the $a_i$'s are symbols. 
Equivalently, the map $u \mapsto u^R$ is the unique map $R \colon \cA^* \to \cA^*$ satisfying $(uv)^R = v^R u^R$ for all $u,v \in \cA^*$.
We propose two proofs of~\cref{pro:close_reversals}: a standalone one using the same arguments as in the proof of and another one relying on the properties of the tree associated with the numeration.

\begin{proposition}
    \label{pro:close_reversals}
    The set of factors of $\mathbf{t}_{3/2}$ and $\mathbf{t}'$ are closed under reversal,  i.e.,  if a word $u\in\cA^*$ is a factor of one of the words,  so is $u^R$.
\end{proposition}
\begin{proof}[First proof of~\cref{pro:close_reversals}]
    Consider the maps $f_0$ and $f_1$ defined in~\eqref{eq:alternateTM}.
    For $i \in \Z$, we define maps $\Phi_i$ as in \eqref{eq:phii} by
    \begin{equation*}
        \Phi_i(a_0 a_1 \cdots a_{k-1}) = f_{i\bmod{2}}(a_0) f_{{i+1}\bmod{2}}(a_1) \cdots f_{{i+k-1}\bmod{2}}(a_{k-1}),
    \end{equation*}
    for all non-empty words $a_0 a_1 \cdots a_{k-1} \in \cA^*$.
    A direct inspection shows that $\Phi_i(a)$ is a palindrome for all $i \in \Z$ and $a \in \{{\tt 0}, {\tt 1}\}$,  i.e,  $\big(\Phi_i(a)\big)^R = \Phi_i(a)$.
    Therefore, for any $i \in \Z$ and any nonempty word $u = a_0a_1\dots a_{k-1} \in \{{\tt 0}, {\tt 1}\}^*$, we have 
    \begin{align*}
        \big(\Phi_i(u)\big)^R &= 
        \big(\Phi_i(a_0) \Phi_{i+1}(a_1) \cdots \Phi_{i+k-1}(a_{k-1})\big)^R \\ &= 
        \big(\Phi_{i+k-1}(a_{k-1})\big)^R\, \big(\Phi_{i+k-2}(a_{k-2})\big)^R \cdots \big(\Phi_i(a_0)\big)^R \\ &=
        \Phi_{i+k-1}(a_{k-1}) \,\Phi_{i+k-2}(a_{k-2}) \cdots \Phi_i(a_0) \\ &=
        \Phi_{k-1-i}(u^R).
    \end{align*}
    We use this identity to prove closure under reversal.

    Since the words $\Phi_0^\ell({\tt 0})$, $\ell \ge 0$, are arbitrarily long prefixes of $\mathbf{t}_{3/2}$, it suffices to show that $ \big(\Phi_0^\ell({\tt 0})\big)^R$ occurs in $\mathbf{t}_{3/2}$ for all $\ell \ge 0$.
    We again proceed by induction on $\ell\ge 0$.
    For the base case $\ell = 0$,  this is clear because $\Phi_0({\tt 0}) = {\tt 0}$ is a palindrome.
    For the induction step,  assume that $ \big(\Phi_0^\ell({\tt 0})\big)^R$ occurs in $\mathbf{t}_{3/2}$ for some $\ell \ge 0$.
    By the identity above, we have
    $\big(\Phi_0^{\ell+1}({\tt 0})\big)^R = \Phi_{k-1}( \big(\Phi_0^\ell({\tt 0}) \big)^R)$,
    where $k = |\Phi_0^\ell({\tt 0})|$.
    Now, by the induction hypothesis, $ \big(\Phi_0^\ell({\tt 0}) \big)^R$ occurs in $\mathbf{t}_{3/2}$, and thus, by \Cref{lem:parity_recurrence}, it occurs in $\mathbf{t}_{3/2}$ at parity $k-1$.
    This implies that $\Phi_{k-1}( \big(\Phi_0^\ell({\tt 0}) \big)^R)$, and hence $\big(\Phi_0^{\ell+1}({\tt 0})\big)^R$, occurs in $\mathbf{t}_{3/2}$,  as desired.
\end{proof}

For readers familiar with rational base numeration systems and, in particular, with the properties of the numeration tree described in \cref{fig:tree32} \cite{Marsault,RigSti1},  we propose an alternative proof of~\cref{pro:close_reversals}.

\begin{proof}[Second proof of~\cref{pro:close_reversals}]
All finite admissible subtrees (i.e., those determined by the periodic rhythm $(2,1)$ of the vertices degrees in the breadth-first traversal) appear in the tree \cite{Marsault}. In particular,  for each $\ell\ge 0$,  one finds a tree $L_\ell$ of height $\ell$ whose leftmost vertex at each level has degree $2$. One also finds a ``symmetric'' tree $R_\ell$ whose rightmost vertex at each level has degree $2$. This information completely determines the structure of the two subtrees. 

Indeed,  at each level of the subtree $L_\ell$ (resp.,  $R_\ell$), starting from the root, one begins with a vertex of degree $2$ that is the leftmost (resp.,  the rightmost) on that level. Then,  vertices of degree $2$ and $1$ alternate starting from the left (resp.,  right). One can show by induction that at each level, $L_\ell$ and $R_\ell$ have the same number of vertices. Since the construction of each level proceeds from left to right in $L_\ell$ and from right to left in $R_\ell$, then,  at each level, the sequence of vertex degrees in $R_\ell$ is the reversal of that in $L_\ell$.

In \cref{fig:tree_sym}, we have depicted two such subtrees of height~$4$ with root $a,c\in\{{\tt 0},{\tt 1}\}$ (we only record the parity of the vertices). 

The first subtree $L_\ell$,  as shown in the proof of \cref{pro:close_bw}, produces a prefix of $\mathbf{t}_{3/2}$ or $\overline{\mathbf{t}_{3/2}}$ (depending on the parity of the root), which is read on the last level. In \cref{fig:tree_sym}, we see the factor $aa\bar{a}\bar{a}\bar{a}a\bar{a}\bar{a}$ on the last level of the tree depicted on the left. The symmetric subtree $R_\ell$ then produces the reversal word. One can then conclude since the set of factors is closed under bit-wise complement,  as desired.
\begin{figure}[h!t]
  \centering
  \tikzset{
  s_bla/.style = {circle,fill=white,thick, draw=black, inner sep=0pt, minimum size=12pt},
  s_red/.style = {circle,black,fill=gray,thick, inner sep=0pt, minimum size=5pt}
}
\begin{tikzpicture}[->,>=stealth',level/.style={sibling distance = 4cm/#1},level distance = 1cm]
  \node [s_bla] {$a$}
  child {node [s_bla] {$a$}
    child {node [s_bla] {$a$}
      child {node [s_bla] {$a$}
       child {node [s_bla] {$a$}}
       child {node [s_bla] {$a$}}
     }
     child {node [s_bla] {$a$}
       child {node [s_bla] {$\bar{a}$}}
     }
    }
    child {node [s_bla] {$a$}
      child {node [s_bla] {$\bar{a}$}
        child {node [s_bla] {$\bar{a}$}}
        child {node [s_bla] {$\bar{a}$}}
      }
    }
  }
  child {node [s_bla] {$a$}
    child {node [s_bla] {$\bar{a}$}
      child {node [s_bla] {$\bar{a}$}
        child {node [s_bla] {$a$}}
      }
      child {node [s_bla] {$\bar{a}$}
        child {node [s_bla] {$\bar{a}$}}
        child {node [s_bla] {$\bar{a}$}}
        }   
      }
    }
;
\end{tikzpicture}\quad\quad 
\begin{tikzpicture}[->,>=stealth',level/.style={sibling distance = 4cm/#1},level distance = 1cm]
  \node [s_bla] {$c$}
  child {node [s_bla] {$c$}
    child {node [s_bla] {$\bar{c}$}
      child {node [s_bla] {$\bar{c}$}
        child {node [s_bla] {$\bar{c}$}}
        child {node [s_bla] {$\bar{c}$}}
      }
      child {node [s_bla] {$\bar{c}$}
        child {node [s_bla] {$c$}}
      }
    }
    }
  child {node [s_bla] {$c$}
    child {node [s_bla] {$c$}
      child {node [s_bla] {$\bar{c}$}
        child {node [s_bla] {$\bar{c}$}}
        child {node [s_bla] {$\bar{c}$}}
      }
    }
    child {node [s_bla] {$c$}
      child {node [s_bla] {$c$}
        child {node [s_bla] {$\bar{c}$}}
      }
      child {node [s_bla] {$c$}
        child {node [s_bla] {$c$}}
        child {node [s_bla] {$c$}}
      }
    }
  }
;
\end{tikzpicture}
  \caption{Two particular subtrees $L_4$ and $R_4$ occurring in the numeration tree of base $3/2$.}
  \label{fig:tree_sym}
\end{figure}
\end{proof}


\section{Frequencies of letters in $\mathbf{t}_{3/2}$}\label{sec:freq}

Rather than studying global frequencies directly, we analyze filtered frequencies along residue classes modulo powers of $2$.  
This refinement captures the intrinsic dyadic structure of $\mathbf{t}_{3/2}$ and naturally leads to a formulation on the $2$-adic integers $\mathbb{Z}_2$. 
In this compact group setting, desubstitution translates into recurrence relations on Fourier coefficients, allowing us to exploit harmonic analysis to obtain a rigidity result.

We filter the counting of occurrences of a symbol $c\in\cA$ in $\mathbf{t}_{3/2}=(t_n)_{n\ge 0}$ with respect to some index modulo a power of~$2$. For integers $n, N \geq 0$,  $c \in \cA$,  and $k \in \Z$,  we let
\begin{equation*}
    C_n(c,k,N) :=
    \#\bigl\{ 0 \leq i < N : t_i = c,\, i \equiv k \bmod 2^n \bigr\},
\end{equation*}
so $C_n(c,k,N)$ gives the number of $c$'s in the prefix $\mathbf{t}_{3/2}[0,N)$ that are in positions congruent to $k$ modulo $2^n$. 
As an example,  looking back at the length-$30$ prefix of $\mathbf{t}_{3/2}$ in~\cref{eq:t32},  the $11$ positions of ${\tt 0}$'s in this prefix are $0, 1, 5, 11, 14, 20, 21, 22, 23, 26, 29$,  so we obtain $(C_1({\tt 0},k,30))_{0\le k < 2^1}=(5,6)$ and $(C_2({\tt 0},k,30)_{0\le k < 2^2}=(2, 4, 3, 2)$.
To avoid making notation heavier by requiring the use of integer parts, we will allow the third argument of $C_n$ to take real values. This will not affect the convergences under consideration,  as,  for $\alpha\in\mathbb{R}_{\ge 0}$, both differences $C_n(c,k,\lceil\alpha\rceil)-C_n(c,k,\alpha)$ and $C_n(c,k,\alpha)-C_n(c,k,\lfloor\alpha\rfloor)$ belong to $\{0,1\}$. 

\begin{theorem}
    \label{thm:existence_freqs_iff_mun_exist}
    For all $n \ge 0$,  $c \in \cA$,  and $k \in \Z$,  we have
    \begin{equation*}
    \lim_{N \to \infty}
    \frac{C_n(c,k,N)}{N} = \frac{1}{2^{n+1}}.
    \end{equation*}
    In particular, the frequency of ${\tt 0}$ (resp.,   ${\tt 1}$) in $\mathbf{t}_{3/2}$ exists and is equal to $1/2$.
\end{theorem}

It will be convenient to introduce some scaling factor $2/3$ coming from a desubstitution process: applying $\tau$ given in \eqref{eq:tau} to a prefix of $\mathbf{t}_{3/2}$  of length $2n$ yields a word of length $3n$ and conversely, a prefix of length $\ell=3n$ comes from a prefix of length $2\ell/3=2n$. So a {\em desubstitution} of $\mathbf{t}_{3/2}$ is a decomposition of $\mathbf{t}_{3/2}$ into consecutive length-$3$ factors, each equal to the image under $\tau$ of a length-$2$ factor.  See~\cref{fig:desub} below for an illustration. 
Note that, as $N$ tends to infinity, $C_n(c,k,N)/N$ converges if and only if
\[
  \frac{C_{n}\left(c,k, \left(2/3\right)^{n} N \right)}{ \left(2/3\right)^{n}N}
\]
converges. In that case, the two sequences converge to the same limit.

\medskip

{\em Sketch of the proof of~\cref{thm:existence_freqs_iff_mun_exist}.} The proof is structured as follows.  We argue by contradiction and assume that for some $n \ge 0$,  $c \in \cA$, $k \in \Z$ and $\varepsilon > 0$, there exists an infinite set $\Lambda \subseteq \N$ such that, for every $N \in \Lambda$,
\[
  \left| \frac{C_{n}(c,k,\left(2/3\right)^{n}N)}{\left(2/3\right)^{n} N} -\frac{1}{2^{n+1}}\right|>\varepsilon.
\]
The we have the following two arguments to obtain a contradiction.
\begin{enumerate}
\item \cref{lem:main_compacity_argument} enables us to extract an increasing sequence $(N_t)_{t\ge0}\in\Lambda^\N$ for which the following limits exist, for all $n,c,k$, 
  \[
    \lim_{t\to\infty} \frac{C_{n}(c,k,\left(2/3\right)^{n}N_t)}{\left(2/3\right)^{n} N_t} = \mu_{n}(c,k)
  \]
  and we obtain recurrence relations linking the $\mu_{n}(c,k)$'s together.
\item \cref{pro:recrel} shows that such recurrence relations imply that the limit $\mu_{n}(c,k)$ equals $2^{-n-1}$, producing the sought contradiction.
\end{enumerate}
The proof of~\cref{lem:main_compacity_argument} comes directly after its statement; that of~\cref{pro:recrel} is delayed until~\cref{sec: proof of Prop 20 for unicity} as more development is required before in~\cref{sec: abstract harmonic analysis}.

We now focus on~\cref{lem:main_compacity_argument}.
As we will see, the recurrence relations \eqref{eq:des_rec} depend on the word $\mathbf{t}_{3/2}$. They essentially come from the desubstitution process.

\begin{lemma}
\label{lem:main_compacity_argument}
For any infinite set $\Lambda \subseteq \N$,  there exists an increasing sequence $(N_t)_{t \ge 0}$ in $\Lambda$ such that, for all $n \ge 0$, $c \in \cA$ and $k \in \Z$, the following four properties are satisfied.
\begin{enumerate}
    \item The limit
    \[  \mu_n(c,k) = 
      \lim_{t \to \infty}
      \frac{C_{n}\left(c,k, \left(2/3\right)^{n} N_t \right)}{\left(2/3\right)^{n} N_t}
    \]
    exists.
  \item We have $\mu_n(c,k+2^n) = \mu_n(c,k)$.
    \item Let $q_n$ be the inverse of $3$ modulo $2^{n+1}$,  i.e.,  $3q_n = 1 \pmod{2^{n+1}}$.
    Then,
    \begin{equation}\label{eq:des_rec}
    \tfrac{3}{2} \mu_n(c,k) = 
    \mu_{n+1}(c, 2q_nk) + 
    \mu_{n+1}(\bar{c}, 2q_nk-q_n) +
    \mu_{n+1}(c, 2q_nk-2q_n).
    \end{equation}
    \item  We have $\mu_n({\tt 0},k) + \mu_n({\tt 1},k) = 2^{-n}$.  
\end{enumerate}
\end{lemma}

\begin{proof}
    The first item follows from a classical compactness argument. Observe that the set of triplets
    \[
    T=\{(n,c,k) \mid n\in\N,  c\in\cA,  k\in\{0,\ldots,2^n-1\}\}
    \]
    is countable,  so we may enumerate the triplets in $T$,  i.e.,  
    \[
    T= \{ (n_0,c_0,k_0), (n_1,c_1,k_1),  (n_2,c_2,k_2),  \ldots \}.
    \]
    For all integers $j\ge 0$ and $N\ge 1$,  the proportion 
   \[
   a_j(N) := \frac{C_{n_j}\left(c_j,k_j,  \left(2/3\right)^{n_j} N \right)}{\left(2/3\right)^{n_j} N}
    \]
    belongs to $[0,1]$,  so it is bounded. 
    Now let us inductively define a sequence $(A_j)_{j\ge 0}$ of infinite subsets of $\N$ with $A_0 = \Lambda$ and, if $A_{j-1} \subset \N$ is an infinite subset of $\N$ that is already built,  then $A_j \subset A_{j-1}$ is  an infinite set chosen such that the sequence $(a_j(N))_{N\ge 0}$ converges along $A_j$. 
    It exists by the theorem of Bolzano--Weierstrass: any bounded sequence in $[0,1]$ admits a converging subsequence. 
    We obtain a sequence of nested infinite subsets
    \begin{align}
    \label{eq: increasingly nested sets}
    \Lambda= A_0 \supset A_1 \supset A_2 \supset \cdots
    \end{align}
    such that,  for each $j\ge 0$,  the first $j$ sequences $(a_1(N))_{N\ge 0}$,  ...,  $(a_j(N))_{N\ge 0}$ all converge along $A_j$. 
    Now we use a diagonal argument: for each integer $t\in\N$,  define $N_t$ to be the $t$-th element of $A_t$ (in increasing order). 
Now fix any integer $j\ge 0$.
Since $A_t \subset A_j$ for all $t\ge j$ by~\cref{eq: increasingly nested sets},  each integer $N_t$ belongs to $A_j$ for each $t\ge j$.
Therefore,  by the property satisfied by the set $A_j$,  the sequence $(a_j(N_t))_{t\ge 0}$ converges (when $t$ tends to infinity),  since it is eventually equal to a subsequence of the converging sequence along $A_j$.
Since $j$ is arbitrarily fixed,  any triplet $(n,c,k)$ in $T$ corresponds to a converging sequence along $(N_t)_{t\ge 0}$.
The limit $\mu_n(c,k)$  thus exists.
    
    The second item follows from the fact that $C_n(c,k + 2^n,N) = C_n(c,k,N)$ by its very definition.
    
    The third item essentially comes from desubstituting $\mathbf{t}_{3/2}$ as shown in~\cref{fig:desub}.
    \begin{figure}
      \centering
\begin{tikzpicture}[>=latex]
  \draw (-.9,0) rectangle (8,.5);
  \draw (-.3,0) rectangle (.3,.5);
  \draw (-.3,0) rectangle (.9,.5);
  \draw (.9,0) rectangle (4,.5);
  
  \draw (-.9,-1) rectangle (10,-1.5);
  \draw (-.9,-1) rectangle (-.1,-1.5);
  \draw (-.9,-1) rectangle (.7,-1.5);
  \draw (-.9,-1) rectangle (1.5,-1.5);
  \draw (1.5,-1) rectangle (6.6,-1.5);
  
  \node at (0,0.25) {$\tt 00\ 11\ 10 $};
  \node at (2.2,0.25) {$\cdots$};
  \node at (8.5,0.25) {$\cdots$};

  \node at (0.3,-1.25) {$\tt 001\ 110 \ 111$};
  \node at (2.8,-1.25) {$\cdots$};
  \node at (10.5,-1.25) {$\cdots$};
  
  \coordinate (b1) at (-0.6,0);
  \coordinate (c1) at (-0.5,-1);
  \draw[->] (b1) to (c1);
  \coordinate (b2) at (0,0);
  \coordinate (c2) at (0.3,-1);
  \draw[->] (b2) to (c2);
  \coordinate (b3) at (0.6,0);
  \coordinate (c3) at (1.1,-1);
  \draw[->] (b3) to (c3);
  \node[left=3pt] at (-.4,-.5) {$\tau$};

  \draw [decorate, decoration={brace, mirror, amplitude=6pt}]
  (-0.9,-1.6) -- (6.6,-1.6)
  node[midway, below=6pt] {$3m$};
  \draw [decorate, decoration={brace, amplitude=6pt}]
  (-0.9,0.6) -- (4,0.6)
  node[midway, above=6pt] {$2m$};
\end{tikzpicture}
      \caption{Desubtituting $\mathbf{t}_{3/2}$.}\label{fig:desub}
    \end{figure}
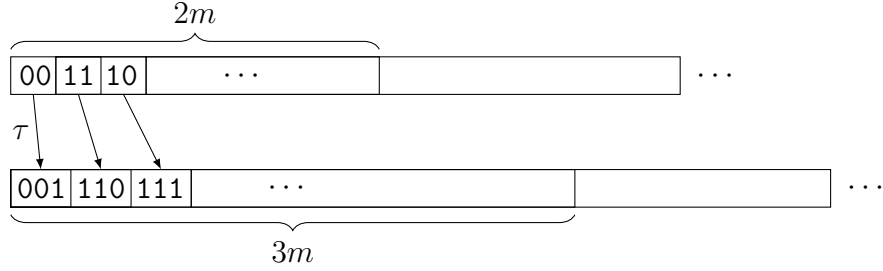
    Consider an arbitrary index congruent to $k$ modulo $2^n$ and of the form
    \[
      3m+r,\quad \text{ with } m\ge 0 \text{ and } r\in\{0,1,2\}.
    \]
    In particular, $k\equiv 3m+r\pmod{2^n}$. Since, $3q_n\equiv 1\pmod{2^n}$, we get $q_nk\equiv m+q_nr\pmod{2^n}$ and finally,
    \begin{equation}\label{eq:mod2m}
       2m \equiv 2q_nk-2q_nr\pmod{2^{n+1}}.
    \end{equation}
    Since $\mathbf{t}_{3/2}$ is a fixed point of $\tau$ defined in~\cref{eq:tau}, we know that 
    \[
    t_{3m}=t_{3m+1}=t_{2m} \quad \text{and} \quad t_{3m+2}=\overline{t_{2m+1}}
    \]
     for all $m\ge 0$. These relations allow us to express $C_n(c,k,N)$.  First,  observe that if $r=0$ (resp.,   $r=1$) in \eqref{eq:mod2m}, then $2m\equiv 2q_nk\pmod{2^{n+1}}$ (resp.,   $2m\equiv 2q_nk-2q_n\pmod{2^{n+1}}$).  Second,  if $r=2$, in \eqref{eq:mod2m}, then $2m\equiv 2q_nk-q_n-3q_n+1 \pmod{2^{n+1}}$, and since $3q_n\equiv 1\pmod{2^{n+1}}$,  we get $2m\equiv 2q_nk-q_n\pmod{2^{n+1}}$.
Hence $C_n(c,k,N)$ is equal to 
    \[
        C_{n+1}(c,2q_nk,\lfloor 2N/3\rfloor) + C_{n+1}(c,2q_nk-2q_n,\lfloor 2N/3\rfloor)  + C_{n+1}(\bar{c},2q_nk-q_n,\lfloor 2N/3\rfloor). \]
      We divide both side by $2N/3$ and we may also get rid of the
      integer parts in the argument of the counting function. The
      attentive reader will note that equality is exact up to one
      unit since,  if $\alpha$ is not an integer, then $C_{n+1}(c,k,\alpha)-C_{n+1}(c,k,\lfloor \alpha\rfloor)\in\{0,1\}$. This bounded difference is not important when passing to the limit.  Hence
      $\frac{3}{2N} C_n(c,k,N)$ is equal to
    \[
       \frac{1}{2N/3}\biggl( C_{n+1}(c,2q_nk,2N/3) + C_{n+1}(c,2q_nk-2q_n, 2N/3)  + C_{n+1}(\bar{c},2q_nk-q_n, 2N/3)
       \biggr). \]
     Now consider this equality along the sequence $((2/3)^n N_t)_{t\ge 0}$ to get the expected relation.
      
    The last item simply translates that along the positions of $\mathbf{t}_{3/2}$ congruent to $k$ modulo $2^n$, we either see a ${\tt 0}$ or a ${\tt 1}$.
\end{proof}

To prove the following result,  we make use of methods from harmonic analysis that we present in~\cref{sec: abstract harmonic analysis}.

\begin{proposition}\label{pro:recrel}
  Assume that a collection of nonnegative real numbers
  \[
    \bigl\{\mu_n(a,k) \mid n \ge 0,\, a \in \cA,\, k \in \Z\bigr\}
  \] satisfies, for all $n \ge 0$,  the following three properties: 
\begin{enumerate}[label=$(P_{\arabic*})$]
    \item For $a \in \cA$, the map $k \mapsto \mu_n(a,k)$ is $2^n$-periodic, i.e.,
    $\mu_n(\cdot,k) = \mu_n(\cdot,k')$ whenever $k \equiv k' \pmod{2^n}$.

    \item Let $q_n$ denote the inverse of $3$ modulo $2^{n+1}$. For all $k \in \Z$ and $a \in \cA$,
    \[  
    \tfrac{3}{2} \, \mu_n(a,k) = \mu_{n+1}(a,\,2q_n k) + \mu_{n+1}(\bar a,\,2q_n k - q_n) + \mu_{n+1}(a,\,2q_n k - 2q_n).
    \]

    \item For all $k \in \Z$,  we have $\mu_n({\tt 0},k)+\mu_n({\tt 1},k) = 2^{-n}$.
\end{enumerate}
Then, for all $n \ge 0$,  $a \in \cA$,  and $k \in \Z$,  we have $ \mu_n(a,k) = 2^{-n-1}$.
\end{proposition}

\begin{remark}
  We observe that condition $(P_1)$ ensures that the value of
  $\mu_{n+1}(a,2q_n k)$ is independent of the representative of $q_n$
  modulo $2^{n+1}$ that one considers.
\end{remark}

\subsection{Some abstract harmonic analysis}
\label{sec: abstract harmonic analysis}

Before proceeding to the proof of~\cref{pro:recrel}, in order to be self-contained and to address an audience who may be more inclined towards a combinatorial reasoning, we briefly recap some key elements of abstract harmonic analysis. A standard reference is \cite{AHA}.


\paragraph{Lift to $2$-adic integers.} The key ingredient is to encode the congruences modulo $2^n$, for all $n$, in a simultaneous way. Let $X=\Z_2$ be the \emph{ring of $2$-adic integers}.  It is classically defined as the inverse limit of the ring homomorphisms
$\Z/2^{n+1}\Z \to \Z/2^n\Z$ given by
reduction modulo powers of $2$.
Formally,
\[
    X =
    \biggl\{(x_n)_{n\ge0} \in \prod_{n\ge0}\Z/2^n\Z
    \mid x_{n+1} \equiv x_n \pmod{2^n} \text{ for all $n\ge 0$}\biggr\}.
\]
As an example,  the sequence starting with $113777(39)\cdots$ belongs to $X$.
Ring operations are defined coordinate-wise: 
for $\mathbf{x}=(x_n)_{n\ge0}$ and $\mathbf{y}=(y_n)_{n\ge0}$ in $X$,  we define $\mathbf{x}+\mathbf{y}$ and $\mathbf{x}\mathbf{y}$ respectively by
\[
    \mathbf{x}+\mathbf{y} = (x_n+y_n \bmod 2^n)_{n\ge0}
    \quad \text{and} \quad
    \mathbf{x}\mathbf{y} = (x_n y_n \bmod 2^n)_{n\ge0}.
\]
With the profinite topology, $X$ is a compact totally disconnected
topological ring.  For any $n\ge 0$,   we let $\pi_n : X \to \Z/2^n\Z$ denote the \emph{canonical
projection} onto the $n$-th coordinate,  i.e.,  for $\mathbf{x}=(x_n)_{n\ge0}\in X$,  we define $\pi_n(\mathbf{x})=x_n$.


\paragraph{The dual group.} 
The so-called {\em Pontryagin dual}  $\widehat X$ of $X$ is the group of \emph{characters},  i.e., continuous homomorphisms from the group $(X,+)$ into the circle group $\{z\in\mathbb{C}: |z|=1\}$ endowed with the complex multiplication. It is canonically isomorphic to 
\[
     \mathbb{Q}[1/2]/\mathbb{Z}
    = \bigl\{\tfrac{a}{2^n} + \mathbb{Z} : a\in\mathbb{Z},\, n\ge0\bigr\}.
  \]
  For a reference, see \cite[p.~113]{AHA} where it is shown that, for a prime~$p$, $\widehat{\mathbb{Z}}_p$ is isomorphic to $\mathbb{Q}_p/\mathbb{Z}_p$. Note that $\mathbb{Q}[1/2]/\mathbb{Z}$ is indeed isomorphic to $\mathbb{Q}_2/\mathbb{Z}_2$ with the isomorphism mapping a representative $\tfrac{a}{2^n} + \mathbb{Z}$ to  $\tfrac{a}{2^n} + \mathbb{Z}_2$.
  
Characters are thus {\em indexed} by dyadic rationals modulo~$1$. 
To explicitly describe these characters, we first define, for all $\mathbf{x}\in X$, a map $P_\mathbf{x} \colon \widehat X \to \widehat X$ by
\[  P_\mathbf{x}\bigl(\tfrac{a}{2^n}+\mathbb{Z}\bigr)
    = \tfrac{\pi_n(x) a}{2^n} + \mathbb{Z},
\]
where $a\in\mathbb{Z}$ and $n\ge0$.

\begin{lemma}
    \label{lem:props_Px}
The previous application satisfies the following properties.
\begin{enumerate}
\item For $\mathbf{x}\in X$,  the application $P_\mathbf{x}$ is well-defined. 
\item For $\mathbf{x}\in X$,  the application $P_\mathbf{x}$ is an additive group morphism,  i.e., $P_\mathbf{x} (r + s) = P_\mathbf{x}(r) + P_\mathbf{x}(s)$ for all $r,s \in \widehat X$.
\item For $\mathbf{x}, \mathbf{y} \in X$,  we have $P_{\mathbf{x}\mathbf{y}}(r) = P_\mathbf{x}(P_\mathbf{x}(r))$ for all $r\in \widehat X$.
\end{enumerate}
\end{lemma}

\begin{proof}
Let us prove that the first item.  Consider two representatives 
\[
\frac{a}{2^n}+\mathbb Z=\frac{b}{2^m}+\mathbb Z
\]
in $\widehat X$ with $a,b\in\mathbb{Z}$ and $m,n\ge0$.
Let $N=\max(n,m)$.  Then
\[
\frac{a2^{N-n}}{2^N}+\mathbb Z
=\frac{b2^{N-m}}{2^N}+\mathbb Z,
\]
which implies
\begin{equation}\label{eq:congruence1}
a2^{N-n}\equiv b2^{N-m}\pmod{2^N}.
\end{equation}
Now,  for any $\mathbf{x}\in X$,  compare the two values of $P_\mathbf{x}$ of these elements after rewriting them with denominator $2^N$,  i.e.
\begin{align*}
P_\mathbf{x}\bigl(\tfrac{a}{2^n}+\mathbb{Z}\bigr)
&=
\frac{\pi_n(\mathbf{x})a}{2^n}+\mathbb Z
=
\frac{\pi_n(\mathbf{x})a\,2^{N-n}}{2^N}+\mathbb Z, \\
P_\mathbf{x}\bigl(\tfrac{b}{2^m}+\mathbb{Z}\bigr)
&=
\frac{\pi_m(\mathbf{x})b}{2^m}+\mathbb Z
=
\frac{\pi_m(\mathbf{x})b\,2^{N-m}}{2^N}+\mathbb Z.
\end{align*}
So it suffices to prove \(
\pi_n(\mathbf{x})a\,2^{N-n}\equiv \pi_m(\mathbf{x})b\,2^{N-m}\pmod{2^N}
\).
Since $\mathbf{x}\in X$, by definition of the inverse limit, the coordinates of $\mathbf{x}\in X$ respectively satisfy $\pi_n(\mathbf{x})\equiv \pi_N(\mathbf{x})\pmod{2^n}$ and $\pi_m(\mathbf{x})\equiv \pi_N(\mathbf{x})\pmod{2^m}$.
Thus,  there exist integers $k,\ell$ such that $\pi_n( \mathbf{x})=\pi_N(\mathbf{x})+2^n k$ and $\pi_m(\mathbf{x})=\pi_N(\mathbf{x})+2^m \ell$.
Multiplying by $a\, 2^{N-n}$ and $b\, 2^{N-m}$ respectively gives
\begin{align*}
\pi_n(\mathbf{x})a\,2^{N-n}&=\pi_N(\mathbf{x})a\,2^{N-n}+a\, 2^N k \equiv \pi_N(\mathbf{x})a \,2^{N-n}\pmod{2^N}, \\
\pi_m(\mathbf{x})b\,2^{N-m}&=\pi_N(\mathbf{x})b\,2^{N-m}+b\, 2^N \ell \equiv \pi_N(\mathbf{x})b\,2^{N-m}\pmod{2^N}.
\end{align*}
To conclude,  multiply \eqref{eq:congruence1} by $\pi_N(\mathbf{x})$ to get
\[
\pi_N(\mathbf{x})a\,2^{N-n}\equiv \pi_N(\mathbf{x})b\,2^{N-m}\pmod{2^N}.
\]
Combining these relations show that $\pi_n(\mathbf{x})a\,2^{N-n}\equiv \pi_m(\mathbf{x})b\,2^{N-m}\pmod{2^N}$,  as desired.

Let us show the second item. 
Let $r,s \in \widehat X$.  As before,  choose a common denominator $2^N$ such that
\[
r=\frac{a}{2^N}+\mathbb Z
\quad \text{ and } \quad
s=\frac{b}{2^N}+\mathbb Z
\]
for $a,b\in\Z$.
Then,  for $\mathbf{x}\in X$,  we have
\[
P_\mathbf{x}(r+s)
=P_\mathbf{x}\!\left(\frac{a+b}{2^N}+\mathbb Z\right)
=\frac{\pi_N(\mathbf{x})(a+b)}{2^N}+\mathbb Z,
\]
while
\[
P_\mathbf{x}(r)+P_\mathbf{x}(s)
=\frac{\pi_N(\mathbf{x})a}{2^N}+\mathbb Z
+\frac{\pi_N(\mathbf{x})b}{2^N}+\mathbb Z
=\frac{\pi_N(\mathbf{x})(a+b)}{2^N}+\mathbb Z.
\]
Thus $P_\mathbf{x}(r+s)=P_\mathbf{x}(r)+P_\mathbf{x}(s)$, and $P_\mathbf{x}$ is a group homomorphism.

Finally, consider the third item of the statement.
Let $r=\frac{a}{2^n}+\mathbb Z$ with $a\in\Z$ and $n\ge 0$.
Fix $\mathbf{x},\mathbf{y}\in X$.
Since multiplication in $X$ is coordinate-wise,  
\(
\pi_n(\mathbf{x}\mathbf{y})=\pi_n(\mathbf{x})\pi_n(\mathbf{y})
\)
 in $\mathbb Z/2^n\mathbb Z$.
Hence
\[
P_{\mathbf{x}\mathbf{y}}(r)=\frac{\pi_n(\mathbf{x}\mathbf{y})a}{2^n}+\mathbb Z=\frac{\pi_n(\mathbf{x})\pi_n(\mathbf{y})a}{2^n}+\mathbb Z=P_\mathbf{x}(P_\mathbf{y}(r)).
\]
This finishes the proof.
\end{proof}


\paragraph{Characters.}
Recall that characters of $ \widehat X$ are indexed by elements of $\mathbb{Q}[1/2]/\mathbb{Z}$. We now describe them. For all $r \in \mathbb{Q}[1/2]/\mathbb{Z}$, we define the map $\chi_r \colon X \to \mathbb{C}$ by
\[
  \chi_r:\mathbf{x}\mapsto e(P_\mathbf{x}(r)),
\]
where $e(\alpha)=\exp(2\pi i\, \alpha)$ is the classical complex exponentiation. These maps are indeed group homomorphisms. 
For $r=\frac{a}{2^n}+\mathbb Z$ with $a\in\Z$ and $n\ge 0$ and all $\mathbf{x},\mathbf{y}\in X$,  we have
\[
  \chi_r(\mathbf{x} + \mathbf{y}) = e\left( \frac{\pi_n(\mathbf{x}+\mathbf{y})a}{2^n}\right)=
  \chi_r(\mathbf{x}) \chi_r(\mathbf{y}),
\]
since $\pi_n(\mathbf{x}+\mathbf{y})\equiv\pi_n(\mathbf{x})+\pi_n(\mathbf{y})\pmod{2^n}$. 
They are useful in our setting because we will apply them to some averaging operator~$\Lop$. 
These are the continuous characters of $X$. 
Further, we check using~\cref{lem:props_Px} that
\begin{equation}
    \label{eq:chi_product}
    \chi_r(\mathbf{x}\mathbf{y}) = e(P_{\mathbf{x}\mathbf{y}}(r))= e(P_{\mathbf{y}\mathbf{x}}(r))= e(P_\mathbf{y}(P_\mathbf{x}(r)))= \chi_{P_\mathbf{x}(r)}(\mathbf{y})
\end{equation}
for all $r \in \mathbb{Q}[1/2]/\mathbb{Z}$ and all $\mathbf{x},\mathbf{y}\in X$.


\paragraph{Pontryagin duality.}
Let $\mu$ be the \emph{normalized Haar measure} on $X$,  which is a \emph{Borel probability measure}\footnote{Note that every locally compact group admits a Haar measure that is unique up to a scalar constant.  In addition,  every normalized Haar measure is a Borel probability measure,  but not every Borel probability measure is Haar.}.
Pontryagin duality (Plancherel theorem for compact abelian groups)~\cite{Tao-Note}
 states that $\{\chi_r : r\in\widehat X\}$ is an orthonormal basis of the Hilbert space 
$L^2(X,\mu)$ \cite[Cor.~4.27]{AHA}. Thus, every $\delta \in L^2(X,\mu)$ has a unique \emph{Fourier expansion}
\begin{equation}
    \label{eq:fourier_decomposition}
    \delta(\mathbf{x}) =
    \sum_{r\in\widehat X} \hat\delta(r)\,\chi_r(\mathbf{x}),
\end{equation}
where $\hat\delta(r) \in \mathbb{C}$ are the {\em Fourier coefficients} of $\delta$.
This expansion satisfies {\em Parseval's identity} \cite[Prop.~4.25, Thm.~4.26]{AHA} that reads
\begin{equation}
    \label{eq:parseval}
    \|\delta\|_2^2 := 
    \int_X |\delta(x)|^2 \, \mathrm{d}\mu(\mathbf{x}) = 
    \sum_{r \in \widehat X}
    |\hat\delta(r)|^2.
\end{equation}

\subsection{Proof of~\cref{pro:recrel}}
\label{sec: proof of Prop 20 for unicity}

\begin{proof}[Proof of~\cref{pro:recrel}] We split the proof into several main steps. 
  
{\bf Reformulating via a difference function.}  Define $\delta_n : X \to \mathbb{R}$ by
\[
    \delta_n(\mathbf{x})
    := 2^n\bigl(\mu_n({\tt 0},\pi_n(\mathbf{x})) - \mu_n({\tt 1},\pi_n(\mathbf{x}))\bigr),
    \quad \text{for $\mathbf{x}\in X$.}
\]
Observe that $\delta_n\equiv 0$ (is identically zero) if and only if $\mu_n({\tt 0},k) =\mu_n({\tt 1},k)$ for all residue classes $k$ modulo~$2^n$. 
Hence, thanks to $(P_3)$, $\delta_n\equiv 0$ if and only if $\mu_n(a,k) = 2^{-n-1}$ for all $n,a,k$.
Thus, our goal is to prove that the hypothesis on $\mu_n$ forces $\delta_n \equiv 0$.

{\bf Uniform Bound on $\delta_n$.} By $(P_3)$,  we have
\begin{equation}\label{eq:bound_delta}
|\delta_n(\mathbf{x})| \le 1,\quad \text{ for all }\mathbf{x}\in X.  
\end{equation}
Indeed, $\mu_n(a,k)$ is nonnegative and by $(P_3)$, $\mu_n({\tt 0},k)+\mu_n({\tt 1},k)=2^{-n}$. So each is between $0$ and $2^{-n}$. Hence, the difference $\mu_n({\tt 0},\pi_n(x)) - \mu_n({\tt 1},\pi_n(x))$ is in $[-2^{-n},2^{-n}]$. The conclusion follows multiplying by $2^n$.

{\bf Translating the recurrence into an operator equation.} We now convert the recurrence $(P_2)$ into a functional equation for $\delta_n$,  which is relevant thanks to the lift to the $2$-adic integers. Let $\mathbf{3}= (3)_{n\ge0} \in X$; because each coordinate is odd, it has a multiplicative inverse
$\mathbf{q}=(q_n)_{n\ge0}$ in $X$. The recurrence of $(P_2)$ becomes
\[
    \delta_n(\mathbf{x}) =
    \sum_{0\le j < 3} c_j \,\delta_{n+1}(2 \mathbf{q} \mathbf{x} - j \mathbf{q}),
\]
where $c_0 = -c_1 = c_2 = 1/3$. Indeed, multiply the difference $\frac{3}{2}\mu_n({\tt 0},k)-\frac{3}{2}\mu_n({\tt 1},k)$ by $2^{n+1}$ and apply $(P_2)$. 
Since $\delta_{n}$ is obtained as a linear combination of $\delta_{n+1}$, this motivates defining an operator $\Lop$ that acts on functions
$\delta : X \to \mathbb{C}$ by
\[  (\Lop\delta)(\mathbf{x}) =
    \sum_{0\le j<3} c_j \,\delta(2 \mathbf{q} \mathbf{x} - j \mathbf{q}).
\]
Then, $\delta_n=\Lop\delta_{n+1}$ and iterating, 
\[
    \delta_n = \Lop^k \delta_{n+k}
    \quad\text{for all $n,k\ge0$.}
\]
If we can show that repeated applications of $\Lop$ shrink any function in a suitable norm, then the only way $\delta_n = \Lop^k \delta_{n+k}$  can hold, for all $k$, is if $\delta_n$ is zero.

Let $\mathbf{2} = (2)_{n\ge0} \in X$.
Note that $P_{\mathbf{2}}$ is everywhere $2$-to-$1$, and that both $P_{\mathbf{3}}$ and $P_{\mathbf{q}}$ are bijective. Indeed, the kernel of $P_{\mathbf{2}}$ is made of $0+\mathbb{Z}$ and $\frac12+\mathbb{Z}$. Note that 
$P_{\mathbf{3}}(\frac{a}{2^n}+\mathbb{Z})=\frac{b}{2^n}+\mathbb{Z}$ if and only if $3a\equiv b\pmod{2^n}$. Since $3$ is invertible modulo $2^n$, the latter congruence has a unique solution. Similar argument applies for $P_{\mathbf{q}}$.

{\bf Action on Fourier coefficients.} For $r \in \widehat X$,
\[
    (\Lop\chi_r)(\mathbf{x})
    = \sum_{0\le j<3} c_j\,  \chi_r(2\mathbf{q} \mathbf{x} - j \mathbf{q}).
\]
Using that $\chi_r$ is a character and \eqref{eq:chi_product},
\[
    \chi_r(2\mathbf{q} \mathbf{x} - j\mathbf{q}) = \chi_r(2\mathbf{q} \mathbf{x}) \chi_r(- j\mathbf{q}) = 
    \chi_{P_{2\mathbf{q}}(r)}(\mathbf{x}) \, \chi_r(-j\mathbf{q}).
\]
Since the first factor on the right-hand side is independent of $j$, we write 
\[
    (\Lop\chi_r)(\mathbf{x})
    = M(r)\, \chi_{P_{2\mathbf{q}}(r)}(\mathbf{x}),
\]
where
\begin{equation}\label{eq:defM}
    M(r)
    := \sum_{0\le j<3} c_j\, \chi_r(-j\mathbf{q})
    = \frac{1}{3}\bigl(1 - e(-P_{\mathbf{q}}(r)) + e(-P_{\mathbf{q}}(r))^2\bigr).
  \end{equation}
For the last term, $\chi_r(-2\mathbf{q})=(\chi_r(-\mathbf{q}))^2$ since it is a character. 
Roughly speaking, this relation can be interpreted as $\Lop$ sending a ``pure frequency'' $\chi_r$ to another frequency $\chi_{P_{2\mathbf{q}}(r)}$ and the amplitude is multiplied by $M(r)$. 
Thus, using~\eqref{eq:fourier_decomposition} --- the operator $\Lop$ can pass through the infinite Fourier sum because the series converges in 
$L^2(X,\mu)$ and $\Lop$ is a bounded linear operator on $L^2(X,\mu)$; bounded operators commute with limits in norm --- we get
\[
    \Lop\delta
    = \sum_{r\in\widehat X}
    \hat\delta(r)\,M(r)\,\chi_{P_{2\mathbf{q}}(r)}.
\]
Since $P_{2\mathbf{q}}$ is everywhere $2$-to-$1$, each character $\chi_s$ appears twice in the sum.
Let us group terms as
\[
    \Lop\delta
    = \sum_{s\in\widehat X}
    \Bigl(\sum_{r\in P_{2\mathbf{q}}^{-1}(s)}
      M(r)\hat\delta(r)\Bigr)\, \chi_s.
\]
We get, as the Fourier expansion is unique,
\begin{equation}
    \label{eq:recu_fourier_coeffs}
    \widehat{(\Lop\delta)}(s) = 
    \sum_{r\in P_{2\mathbf{q}}^{-1}(s)}
    M(r)\hat\delta(r)
\end{equation}
for all $s \in \widehat X$ and $\delta \in L^2(X,\mu)$. 

{\bf Iterating the operator in Fourier space.} Let $\delta \in L^2(X,\mu)$ and $k \ge 1$.
Define 
\[  M^{(k)}(r) := \prod_{0 \le \ell < k}
    M(P_{2\mathbf{q}}^\ell(r))
    \enspace 
    \text{for $r \in \widehat X$.}
\]
Then, for every $s \in \widehat X$,
\[  \widehat{(\Lop^k\delta)}(s) = 
    \sum_{r\in P_{2\mathbf{q}}^{-k}(s)}
    M^{(k)}(r) \, \hat\delta(r).
  \]
For an increased readability, the proof of this fact presented as~\cref{lem:recu_fourier_coeffs_k_steps} is given outside this proof. This result lemma and Parseval's identity \eqref{eq:parseval} give
\[  \|\Lop^k\delta\|_2^2 = 
    \sum_{s\in\widehat X}
    \bigg|\sum_{r\in P_{2\mathbf{q}}^{-k}(s)} 
    M^{(k)}(r) \, \hat\delta(r)\bigg|^2.
\]
Using Cauchy--Schwarz in the inner sum,
\[
    \|\Lop^k\delta\|_2^2\le
    \sum_{s \in \widehat X}
    \Bigl(\sum_{r\in P_{2\mathbf{q}}^{-k}(s)} 
    |M^{(k)}(r)|^2\Bigr)
    \Bigl(\sum_{r\in P_{2\mathbf{q}}^{-k}(s)} |\hat\delta(r)|^2\Bigr).
\]
Define
\begin{align}
\label{eq:def xik}
\zeta_k(s) = 
    \sum_{r\in P_{2\mathbf{q}}^{-k}(s)} 
    |M^{(k)}(r)|^2
    \quad \text{and} \quad 
    \|\zeta_k\|_\infty = 
    \sup_{s \in \widehat X}
    |\zeta_k(s)|.
\end{align}
Then
\[
    \|\Lop^k\delta\|_2^2 \le
    \sum_{s\in\widehat X} 
    \zeta_k(s) \, 
    \sum_{r \in P_{2\mathbf{q}}^{-k}}
    |\hat\delta(r)|^2 \le 
    \|\zeta_k\|_\infty \, 
    \sum_{s \in \widehat X}
    \sum_{r \in P_{2\mathbf{q}}^{-k}(s)}
    |\hat\delta(r)|^2.
\]
Since $\{P_{2\mathbf{q}}^{-k}(s) : s \in \widehat X\}$ is a partition of $\widehat X$, the last double sum corresponds, by Parseval's identity \eqref{eq:parseval}, to $\|\delta\|^2$.
Therefore,
\begin{equation}
    \label{eq:norm_contraction:general_k}
    \|\Lop^k\delta\|_2^2 \le
    \|\zeta_k\|_\infty \, 
    \|\delta\|_2^2
\end{equation}
for all $\delta \in L^2(X,\mu)$ and $k \ge 1$.

{\bf Conclusion.} It is enough to show that $\|\zeta_2\|_\infty < 1$. Indeed, the contracting property $\|\zeta_2\|_\infty < 1$ (proved later in~\cref{sec:33}) implies that $\mu_n(a,k) = 2^{-n-1}$ is the unique solution to the linear recurrence relation: Since $\delta_n = \Lop^{2k}\delta_{n+2k}$ for all $n,k \ge 0$,  we obtain from  \eqref{eq:norm_contraction:general_k} applying the $\Lop^2$-contraction repeatedly
\[
    \|\delta_n\|_2^2
    = \|\Lop^2\delta_{n+2}\|_2^2
    \le \|\zeta_2\|_\infty \|\delta_{n+2}\|_2^2
    \le \dotsb \le \|\zeta_2\|_\infty^k \|\delta_{n+2k}\|_2^2.
\]
As observed in \eqref{eq:bound_delta}, we have the uniform bound $|\delta_m(\mathbf{x})|\le1$ for all $m \ge 0$ and $\mathbf{x} \in X$, so $\|\delta_m\|_2^2 \le 1$ for all $m$ since $\mu$ is a probability measure.
Hence, $\|\delta_n\|_2^2 \le \|\zeta_2\|_\infty^k$.
Taking the limit $k \to \infty$, we obtain that $\delta_n = 0$ identically as $\|\zeta_2\|_\infty < 1$.
Note that in $L^2(X,\mu)$ having a norm zero means that the function is $0$ almost everywhere but here $\delta_n$ is locally constant, so it is identically $0$, which was our goal. This means that $\mu_n({\tt 0},k) = \mu_n({\tt 1},k)$ for all $n \ge 0$ and $k \in \Z$. Using $(P_3)$, we conclude
\[
    \mu_n({\tt 0},k)=\mu_n({\tt 1},k)=2^{-n-1},
\]
as desired.
\end{proof}

We now prove the identity in the penultimate step of the previous proof.

\begin{lemma}
\label{lem:recu_fourier_coeffs_k_steps}
Let $\delta \in L^2(X,\mu)$ and $k \ge 1$. With $M(r)$ defined in \eqref{eq:recu_fourier_coeffs}, we let  
\[  M^{(k)}(r) := \prod_{0 \le \ell < k}
    M(P_{2\mathbf{q}}^\ell(r))
    \enspace 
    \text{for $r \in \widehat X$.}
\]
Then, for every $s \in \widehat X$,
\[  \widehat{(\Lop^k\delta)}(s) = 
    \sum_{r\in P_{2\mathbf{q}}^{-k}(s)}
    M^{(k)}(r) \, \hat\delta(r).
\]
\end{lemma}
\begin{proof}
Proceed by induction. The case $k=1$ is given by \eqref{eq:recu_fourier_coeffs}. Assume now that the formula holds for some $k\ge 1$, i.e.,
\[
\widehat{\Lop^k\delta}(t)
=
\sum_{u\in P_{2\mathbf{q}}^{-k}(t)} M^{(k)}(u)\,\widehat{\delta}(u)
\quad\text{for all } t\in\widehat X.
\]
Applying \eqref{eq:recu_fourier_coeffs} to $\Lop^k\delta$, we obtain
\[
\widehat{\Lop^{k+1}\delta}(s)=\widehat{\Lop(\Lop^{k}\delta)}(s)
=
\sum_{t\in P_{2\mathbf{q}}^{-1}(s)} M(t)\,\widehat{\Lop^k\delta}(t).
\]
Substituting the induction hypothesis yields
\[
\widehat{\Lop^{k+1}\delta}(s)
=
\sum_{t\in P_{2\mathbf{q}}^{-1}(s)} M(t)
\sum_{u\in P_{2\mathbf{q}}^{-k}(t)} M^{(k)}(u)\,\widehat{\delta}(u).
\]
We now reindex the sum.  
The conditions $t\in P_{2\mathbf{q}}^{-1}(s)$ and $u\in P_{2\mathbf{q}}^{-k}(t)$ are equivalent to
\[
P_{2\mathbf{q}}^{k+1}(u)=s,
\]
with $t=P_{2\mathbf{q}}^k(u)$. Hence the pairs $(t,u)$ are in bijection with
$u\in P_{2\mathbf{q}}^{-(k+1)}(s)$, and the above expression becomes
\[
\widehat{\Lop^{k+1}\delta}(s)
=
\sum_{u\in P_{2\mathbf{q}}^{-(k+1)}(s)}
M\!\big(P_{2\mathbf{q}}^k(u)\big)\,M^{(k)}(u)\,\widehat{\delta}(u).
\]
Finally, by definition of $M^{(k)}$,
\(
M\!\big(P_{2\mathbf{q}}^k(u)\big)\,M^{(k)}(u)=M^{(k+1)}(u)
\).
\end{proof}


\subsection{Proving that $\|\zeta_2\|_\infty < 1$}\label{sec:33}
We consider a multiplication by $3^k$ to cancel the $q$'s occurring in the iterates of $P_{\mathbf{q}}$ and $P_{2\mathbf{q}}$. For $r \in \widehat X$, we thus define
\begin{equation}
    \label{eq:defi_widetilde_M_k}
    \widetilde{M}^{(k)}(r) :=
    M^{(k)}(3^k r).
  \end{equation}
  \begin{lemma} We have 
\[     \widetilde{M}^{(k)}(r)=\frac{1}{3^k}
    \prod_{0 \le \ell < k}
    \bigl(1 - e(-2^\ell \, 3^{k-\ell-1} r) + e(-2^{\ell+1} \, 3^{k-\ell-1} r)\bigr).
  \]
\end{lemma}

\begin{proof} By definition, we have 
\[
\widetilde M^{(k)}(r)=\prod_{0\le \ell<k} M\!\bigl(P_{2\mathbf{q}}^\ell(3^k r)\bigr).
\]
Since $2\mathbf{q}=\mathbf{2}\,\mathbf{q}$ in $X$, we have
\(
P_{2\mathbf{q}}=P_{\mathbf{2}}\circ P_{\mathbf{q}} 
\).
Iterating and using commutativity on $\widehat X$ gives, for every $\ell\ge 0$, 
\[
P_{2\mathbf{q}}^\ell(3^k r)=2^\ell q^\ell 3^k r.
\]
Since $\mathbf{q}$ is the $2$-adic inverse of $\mathbf{3}$, thus 
\[
P_{2\mathbf{q}}^\ell(3^k r)=2^\ell 3^{k-\ell} r \quad \text{in }\widehat X.
\]

Next recall from \eqref{eq:defM} that 
\[
M(t)=\frac13\Bigl(1-e(-P_{\mathbf{q}}(t))+e(-2P_{\mathbf{q}}(t))\Bigr),
\]
Using again $P_{\mathbf{x}\mathbf{y}}=P_ \mathbf{x}\circ P_\mathbf{y}$ and that $P_{\mathbf{q}}$ acts by multiplication by $q$,
we have
\[
P_{\mathbf{q}}\bigl(P_{2\mathbf{q}}^\ell(3^k r)\bigr)
=
q\cdot P_{2\mathbf{q}}^\ell(3^k r)
=
q\cdot (2^\ell 3^{k-\ell} r)
=
2^\ell 3^{k-\ell-1} r
\quad \text{in }\widehat X.
\]
Therefore,
\begin{equation}\label{eq:M-on-iterate}
M\!\bigl(P_{2\mathbf{q}}^\ell(3^k r)\bigr)
=
\frac13\Bigl(
1-e(-2^\ell 3^{k-\ell-1}r)+e(-2^{\ell+1}3^{k-\ell-1}r)
\Bigr).
\end{equation}
and factoring out
$3^{-k}$ gives the desired product formula.
\end{proof}

\begin{lemma}
For any $s \in \mathbb{Q}[1/2] \cap [0,1)$, 
\[  \zeta_k(3^ks) = 
    \sum_{0 \le i < 2^k}
    |\widetilde{M}^{(k)}(\tfrac{s}{2^k} + \tfrac{i}{2^k})|^2.
  \]
\end{lemma}
\begin{proof}
Let $s \in \mathbb Q[1/2] \cap [0,1)$. Then
\[
r \in P_{2\mathbf{q}}^{-k}(3^k s)
\quad\Longleftrightarrow\quad
2^k r \equiv 3^{2k} s \pmod{1}.
\]
Thus there exists $m \in \mathbb Z$ such that
\[
r = \frac{3^{2k}}{2^k} s + \frac{m}{2^k}.
\]
Reducing modulo $1$, the distinct solutions are obtained by taking
$m = 0,1,\dots,2^k-1$, and we obtain
\[
P_{2\mathbf{q}}^{-k}(3^k s)
=
\left\{
\frac{3^{2k}}{2^k} s + \frac{j}{2^k} + \mathbb Z
\;:\;
0 \le j < 2^k
\right\}.
\]
Since $3^{2k}$ is odd, multiplication by $3^{2k}$ permutes the elements of
$\widehat X = \mathbb Q[1/2]/\mathbb Z$. Therefore, recalling the definition in~\cref{eq:def xik},  we obtain
\[
\zeta_k(3^k s)
=
\sum_{r \in P_{2\mathbf{q}}^{-k}(3^k s)} \bigl| \widetilde M^{(k)}(r) \bigr|^2
=
\sum_{0 \le j < 2^k}
\left|
\widetilde M^{(k)}\!\left( \frac{3^{2k}}{2^k} s + \frac{j}{2^k} \right)
\right|^2.
\]
Because we have a function depending
on $r$ only through its value modulo $1$ and invariant under permutation of the
summation index, we may equivalently write
\[
P_{2\mathbf{q}}^{-k}(3^k s)
=
\left\{
\frac{s}{2^k} + \frac{j}{2^k} + \mathbb Z
\;:\;
0 \le j < 2^k
\right\},
\]
which gives the statement.
\end{proof}
  
Consequently,  since $\mathbb{Q}[1/2]$ is dense in $\mathbb{R}$ and $\widetilde{M}^{(k)}$ continuous,  we obtain
\[  \sup_{s \in \widehat X}
    \zeta_k(3^ks) = 
    \sup_{s \in \mathbb{R}}
    \sum_{0 \le j < 2^k}
    |\widetilde{M}^{(k)}(\tfrac{s}{2^k} + \tfrac{j}{2^k})|^2.
\]
Furthermore,  multiplication by powers of $3$ permutes the elements of $\widehat X$, so recalling~\cref{eq:def xik} we get
\begin{equation}
    \label{eq:formula_norm_zeta_k}
    \|\zeta_k\|_\infty = 
    \sup_{s \in \mathbb{R}}
    \sum_{0 \le j < 2^k}
    |\widetilde{M}^{(k)}(\tfrac{s}{2^k} + \tfrac{j}{2^k})|^2.
\end{equation}
We use this formula to compute $\|\zeta_2\|_\infty$.

Set $k=2$. 
For any $r\in \widehat X$, we can compute using \eqref{eq:defi_widetilde_M_k}
\[
\widetilde M^{(2)}(r)
=\frac{1}{9}\bigl(1-e(-3r)+e(-6r)\bigr)\bigl(1-e(-2r)+e(-4r)\bigr)
=\sum_{0\le \ell \le 10} g_\ell \, e(-\ell r),
\]
where
\[
(g_0,g_1,g_2,g_3,g_4,g_5,g_6,g_7,g_8,g_9,g_{10})
=\tfrac{1}{9}(1,0,-1,-1,1,1,1,-1,-1,0,1).
\]
Using the identity $|\xi|^2=\xi\,\xi^*$ for $\xi\in\mathbb{C}$ (where we let $\xi^*$ denote the complex conjugate of $\xi\in\mathbb{C}$), we get
\begin{align*}
\sum_{0 \le j < 4}\left|\widetilde M^{(2)}\!\left(\tfrac{s}{4}+\tfrac{j}{4}\right)\right|^2
&= \sum_{0 \le j < 4}\sum_{0\le \ell,  \ell'\le 10} g_{\ell} g_{\ell'}\,
e\!\left(-(\ell-\ell')\!\left(\tfrac{s}{4}+\tfrac{j}{4}\right)\right) \\ 
&= \frac{4}{81}\Bigl(9+e(s)+e(-s)-2e(2s)-2e(-2s)\Bigr).
\end{align*}
Since \(2\cos(t) = e(t)+e(-t)\), we may rewrite the last expression as 
\[
\sum_{0 \le j < 4}\left|\widetilde M^{(2)}\!\left(\tfrac{s}{4}+\tfrac{j}{4}\right)\right|^2 =
\frac{4}{81}\Bigl(9+2 \cos(2\pi s) - 4 \cos(4\pi s)\Bigr).
\]
Thus, by \eqref{eq:formula_norm_zeta_k},
\[
\|\zeta_2\|_\infty =
\sup_{s \in \mathbb{R}}
\frac{4}{81}\Bigl(9+2\cos(2\pi s) - 4 \cos(4\pi s)\Bigr) \le 
\frac{4}{81}\Bigl(9+2+4\Bigr) = \frac{20}{27} < 1,
\]
which proves our claim.

\section{Concluding remarks}
\label{sec: conclusion}

This article opens the way to many challenging questions. First, the Thue--Morse word $\mathbf{t}_{3/2}$ can be generalized in two ways: we may consider the sum-of-digits modulo an integer $m\ge 2$ instead of $2$ and/or we may change the numeration system to other rational bases \cite{GolafshanRigo,Starosta}. 

In the first option,  the generalized Thue--Morse word corresponding to the sum-of-digits of $3/2$-expansions modulo $m\ge 2$ is obtained from the $2$-block substitution
\[
  ab \mapsto a(a+2\bmod{m})(b+1\bmod{m}).
\]
For instance, with $m=4$, the prefix of the word is ${\tt 023310131130132311130200132130}\cdots$.
In that setting, following the same lines,  the recurrence in~\eqref{eq:des_rec} becomes
\[
    \tfrac{3}{2} \mu_n(c,k) = 
    \mu_{n+1}(c, 2q_nk) + 
    \mu_{n+1}(c+1, 2q_nk-q_n) +
    \mu_{n+1}(c+2, 2q_nk-2q_n).
  \]
  for $n\ge 0$ and $c\in\mathbb{Z}/m\mathbb{Z}$.
As in \cref{pro:recrel}, we have a collection of nonnegative real numbers
  \[
    \bigl\{\mu_n(c,k) \mid n \ge 0,\, c \in \mathbb{Z}/m\mathbb{Z},\, k \in \Z\bigr\}
  \]
  satisfying properties similar to those in \cref{pro:recrel}. The technical problem to overcome is that the difference function~$\delta_n:X\to\mathbb{R}$ appearing in the proof of \cref{pro:recrel} has to be replaced by a function taking into account the different pairs of letters. We believe that the frequency of each symbol exists and is equal to $1/m$ (this fact is also supported by computer experiments).

Another direction of research is as follows.
  Let $\mathbf{x}$ be a fixed point of a $q$-block substitution where the image of each $q$-block has length~$p$. Broader questions of interest are the following:
  \begin{enumerate}
  \item Does the frequency of a symbol $a$ in $\mathbf{x}$ exist?
  \item If it exists, is it always a rational number?
  \item Can we easily determine if such an infinite word~$\mathbf{x}$ is uniformly recurrent? Indeed, in the classical setting, a $k$-automatic sequence is uniformly recurrent if and only if its underlying $k$-uniform morphism is primitive \cite{Durand}. We do not have such a result for alternating morphisms.
  \end{enumerate}

  Answering the first two questions would help to address the problem of determining the frequency of factors (and not only symbols). Indeed, it is easy to see that the sliding block code --- with window size $\ell\ge 1$ --- of a fixed point of a $q$-block substitution where the image of each $q$-block has length~$p$, is also generated by a substitution with the same property. As an example, consider a window of size~$\ell=2$. Hence,  blocks of size $2$ represent integers in $\{0,1,2,3\}$.  The sliding block code of $\mathbf{t}_{3/2}$ then starts with
  \[
    \mathbf{c}=(c_n)_{n\ge 0}={\tt 013321333321321333320001321320}\cdots,
  \]
 where $c_n=\val_2(t_nt_{n+1})=2t_n+t_{n+1}$ for all $n\ge 0$. 
    This word is a fixed point of the substitution
    \begin{align*}
      &{\tt 00}\to {\tt 012},\ 
      {\tt 01}\to {\tt 013},\ 
      {\tt 12}\to {\tt 000},\ 
      {\tt 13}\to {\tt 001},\\
      &{\tt 20}\to {\tt 332},\ 
      {\tt 21}\to {\tt 333},\ 
      {\tt 32}\to {\tt 320},\ 
      {\tt 33}\to {\tt 321}.
    \end{align*}
Note that for any three consecutive symbols $t_it_{i+1}t_{i+2}$, if $t_{i+1}={\tt 0}$ (resp.,   $t_{i+1}={\tt 1}$), then $c_i=\val_2(t_i{\tt 0})$ is even (resp.,  odd) and $c_{i+1}=\val_2({\tt 0}t_{i+1})$ is in $\{{\tt 0},{\tt 1}\}$ (resp.,  $c_{i+1}=\val_2({\tt 1}t_{i+1})$ is in $\{{\tt 2},{\tt 3}\}$). This explains why only eight $2$-blocks may appear.
    
    Experiments show that in this word, the frequencies of {\tt 0} and {\tt 3} are equal to $3/10$ whereas those of {\tt 1} and {\tt 2} are equal to $2/10$. This would mean that the frequencies of {\tt 00} and {\tt 11} (resp.,  {\tt 01} and {\tt 10}) in $\mathbf{t}_{3/2}$ are equal to $3/10$ (resp.,  $2/10$). In particular, this shows that the argument in the proof of \cref{pro:recrel}, that $\delta_n$ is identically zero,  fails since $\mu_n({\tt 0},\pi_n(\mathbf{c}))-\mu_n({\tt 1},\pi_n(\mathbf{c}))$ should not vanish.

    This corresponds to the values described by Dekking for $\mathbf{t}'$ \cite{DekTM}. Assume we want to compute the frequency of {\tt 01} in $\mathbf{t}'$. By \cref{lem:morphic_image}, each {\tt 0} (resp.,  {\tt 1}) in the prefix of $\mathbf{t}_{3/2}$ of length $n$ produces with $\varphi$ one occurrence of {\tt 01} in the prefix of $\mathbf{t}'$ of length $3n$. Moreover, each {\tt 01} in $\mathbf{t}_{3/2}$ produces an extra {\tt 01} in $\mathbf{t}'$. Therefore, assuming the existence of frequencies in $\mathbf{t}_{3/2}$ yields
    \[
      \freq_{\mathbf{t'}}({\tt 01})=\frac13 \left(\freq_{\mathbf{t}_{3/2}}({\tt 0})+\freq_{\mathbf{t}_{3/2}}({\tt 1})+\freq_{\mathbf{t}_{3/2}}({\tt 01})\right)=\frac13+\frac{2}{30}=\frac{4}{10}.
    \]
    
    The open problems addressed above do not allow us to tackle the conjectures related to the Oldenburger--Kolakoski word~$\mathbf{k}$. Indeed, in this article, we consider substitutions whose images have constant length. For the word~$\mathbf{k}$, the substitution~$\kappa$ is non-uniform (\cref{ssec:per}). This has an important consequence: the positions of letters modulo powers of $2^n$ do not control the iteration of the substitution. In other words, unlike in the constant-length setting, congruence classes modulo $2^n$ are no longer sufficient to describe how letters evolve under iteration. Even if one replaces the naive notion of parity by an appropriate one --- namely, the state of the iterated transducer in which a given letter occurs --- one still needs an equation of the form $\mu_n({\tt 1},k)+\mu_n({\tt 2},k)=2^{-n}$ in order to apply Fourier analysis. Without such a relation, the $L^2$-norms arising in the Fourier approach become unbounded, preventing any meaningful spectral control. This condition is in fact equivalent to an equidistribution of these generalized ``parities''. However, such a property already appears to be close to uniform recurrence of the word itself.

    Nonetheless, we hope that the new techniques developed here could provide new research directions toward solving the conjectures.
\section*{Acknowledgments}

We thank Gabriele Fici who shared his notes on this problem and Mai-Linh Trân Công for her proofreading of this paper.  

Bastián Espinoza is supported by ULiège's Special Funds for Research, IPD-STEMA Program.
Michel Rigo is supported by the FNRS Research grant T.196.23 (PDR). 
Manon Stipulanti is an FNRS Research Associate supported by the Research grant 1.C.104.24F.

\bibliographystyle{plain}
\bibliography{biblio}

@article{GolafshanRigo,
    AUTHOR = {Golafshan, Mehdi and Rigo, Michel and Whiteland, Markus A.},
     TITLE = {Computing the {$k$}-binomial complexity of generalized
              {T}hue-{M}orse words},
   JOURNAL = {J. Combin. Theory Ser. A},
    VOLUME = {220},
      YEAR = {2026},
     PAGES = {Paper No. 106152, 54},
      ISSN = {0097-3165,1096-0899},
   MRCLASS = {68R15},
  MRNUMBER = {5002220},
       DOI = {10.1016/j.jcta.2025.106152},
       URL = {https://doi.org/10.1016/j.jcta.2025.106152},
}

@article{Starosta,
    AUTHOR = {Starosta, {\v S}t{\v e}p{\'a}n},
     TITLE = {Generalized {T}hue-{M}orse words and palindromic richness},
   JOURNAL = {Kybernetika (Prague)},
  FJOURNAL = {Kybernetika},
    VOLUME = {48},
      YEAR = {2012},
    NUMBER = {3},
     PAGES = {361--370},
      ISSN = {0023-5954,1805-949X},
}

@article{Durand,
    AUTHOR = {Durand, Fabien},
     TITLE = {A characterization of substitutive sequences using return
              words},
   JOURNAL = {Discrete Math.},
    VOLUME = {179},
      YEAR = {1998},
    NUMBER = {1-3},
     PAGES = {89--101},
      ISSN = {0012-365X,1872-681X},
       DOI = {10.1016/S0012-365X(97)00029-0},
}

@PhdThesis{Marsault,
  author = 	 {Marsault, Victor},
  title = 	 {{\'E}num{\'e}ration et num{\'e}ration},
  school = 	 {T{\'e}lecom-Paristech},
  year = 	 2015
}

@ARTICLE{1096090,
  author={Cleary,  John G. and Witten,  Ian H. },
  journal={IEEE Transactions on Communications}, 
  title={Data Compression Using Adaptive Coding and Partial String Matching}, 
  year={1984},
  volume={32},
  number={4},
  pages={396--402},
  doi={10.1109/TCOM.1984.1096090}}

@article {MR4596606,
    AUTHOR = {Byszewski, Jakub and Konieczny, Jakub and M\"ullner, Clemens},
     TITLE = {Gowers norms for automatic sequences},
   JOURNAL = {Discrete Anal.},
  FJOURNAL = {Discrete Analysis},
     PAGES = {Paper No. 4, 62},
     YEAR = {2023},
}

@article{BaakeGrimmpaper,
 author = {Baake, Michael and Grimm, Uwe},
 title = {The singular continuous diffraction measure of the {Thue}-{Morse} chain},
 journal = {J. Phys. A, Math. Theor.},
 issn = {1751-8113},
 volume = {41},
 number = {42},
 pages = {6},
 note = {Id/No 422001},
 year = {2008},
 doi = {10.1088/1751-8113/41/42/422001}
}

@article{MullnerS,
 author = {M{\"u}llner, Clemens and Spiegelhofer, Lukas},
 title = {Normality of the {Thue}-{Morse} sequence along {Piatetski}-{Shapiro} sequences. {II}},
 journal = {Isr. J. Math.},
 issn = {0021-2172},
 volume = {220},
 number = {2},
 pages = {691--738},
 year = {2017},
 doi = {10.1007/s11856-017-1531-x}
}

@book {AHA,
    AUTHOR = {Folland, Gerald B.},
     TITLE = {A course in abstract harmonic analysis},
    SERIES = {Textbooks in Mathematics},
   EDITION = {Second},
 PUBLISHER = {CRC Press, Boca Raton, FL},
      YEAR = 2016,
     PAGES = {xiii+305 pp.+loose errata},
      ISBN = {978-1-4987-2713-6}
}

@article {Kolakoski,
    AUTHOR = {Kolakoski, William},
     TITLE = {{A}dvanced {P}roblems: 5304},
   JOURNAL = {Amer. Math. Monthly},
  FJOURNAL = {American Mathematical Monthly},
    VOLUME = {72},
      YEAR = {1965},
    NUMBER = {6},
     PAGES = {673--675},
      ISSN = {0002-9890,1930-0972},
   MRCLASS = {99-04},
  MRNUMBER = {1533328},
       DOI = {10.2307/2313883},
       URL = {https://doi.org/10.2307/2313883},
}

@article {Oldenburger,
    AUTHOR = {Oldenburger, Rufus},
     TITLE = {Exponent trajectories in symbolic dynamics},
   JOURNAL = {Trans. Amer. Math. Soc.},
  FJOURNAL = {Transactions of the American Mathematical Society},
    VOLUME = {46},
      YEAR = {1939},
     PAGES = {453--466},
      ISSN = {0002-9947,1088-6850},
   MRCLASS = {46.3X},
  MRNUMBER = {352},
MRREVIEWER = {Gustav\ A.\ Hedlund},
       DOI = {10.2307/1989933},
       URL = {https://doi.org/10.2307/1989933},
}

@article {DekTM,
    AUTHOR = {Dekking,  Frederik Michel},
     TITLE = {The {T}hue-{M}orse sequence in base 3/2},
   JOURNAL = {J. Integer Seq.},
  FJOURNAL = {Journal of Integer Sequences},
    VOLUME = {26},
      YEAR = {2023},
    NUMBER = {2},
     PAGES = {Art. 23.2.3, 6},
      ISSN = {1530-7638}
}

@article {MR4907985,
    AUTHOR = {Eliahou, Shalom and Verger-Gaugry, Jean-Louis},
     TITLE = {The number system in rational base {$3/2$} and the {$3x+1$}
              problem},
   JOURNAL = {C. R. Math. Acad. Sci. Paris},
  FJOURNAL = {Comptes Rendus Math\'ematique. Acad\'emie des Sciences. Paris},
    VOLUME = {363},
      YEAR = {2025},
     PAGES = {329--336},
      ISSN = {1631-073X,1778-3569},
   MRCLASS = {11B83 (11A63 68R15)},
  MRNUMBER = {4907985},
MRREVIEWER = {Olivier\ Rozier},
       DOI = {10.5802/crmath.662},
       URL = {https://doi.org/10.5802/crmath.662},
}

@article {MR4432963,
    AUTHOR = {Rossi, Luc\'ia and Thuswaldner, J\"org M.},
     TITLE = {A number system with base {$-\frac32$}},
   JOURNAL = {Amer. Math. Monthly},
  FJOURNAL = {American Mathematical Monthly},
    VOLUME = {129},
      YEAR = {2022},
    NUMBER = {6},
     PAGES = {539--553},
      ISSN = {0002-9890,1930-0972},
   MRCLASS = {11A63 (28A80)},
  MRNUMBER = {4432963},
MRREVIEWER = {Li-Xiang\ An},
       DOI = {10.1080/00029890.2022.2061281},
       URL = {https://doi.org/10.1080/00029890.2022.2061281},
}

@article {MR4011502,
    AUTHOR = {Hofer, Roswitha},
     TITLE = {A lower bound on the star discrepancy of generalized {H}alton
              sequences in rational bases},
   JOURNAL = {Proc. Amer. Math. Soc.},
  FJOURNAL = {Proceedings of the American Mathematical Society},
    VOLUME = {147},
      YEAR = {2019},
    NUMBER = {11},
     PAGES = {4655--4664},
      ISSN = {0002-9939,1088-6826},
   MRCLASS = {11K31 (11K38)},
  MRNUMBER = {4011502},
MRREVIEWER = {\'Isabel\ Pirsic},
       DOI = {10.1090/proc/14596},
       URL = {https://doi.org/10.1090/proc/14596},
}

@article {RigSti2,
    AUTHOR = {Rigo, Michel and Stipulanti, Manon},
     TITLE = {Revisiting regular sequences in light of rational base
              numeration systems},
   JOURNAL = {Discrete Math.},
  FJOURNAL = {Discrete Mathematics},
    VOLUME = {345},
      YEAR = {2022},
    NUMBER = {3},
     PAGES = {Paper No. 112735, 23},
      ISSN = {0012-365X,1872-681X},
   MRCLASS = {68R15 (11B85 68Q45)},
  MRNUMBER = {4345966},
MRREVIEWER = {\'Emilie\ Charlier},
       DOI = {10.1016/j.disc.2021.112735},
       URL = {https://doi.org/10.1016/j.disc.2021.112735},
}

@article {AFS,
    AUTHOR = {Akiyama, Shigeki and Frougny, Christiane and Sakarovitch,
              Jacques},
     TITLE = {Powers of rationals modulo 1 and rational base number systems},
   JOURNAL = {Israel J. Math.},
  FJOURNAL = {Israel Journal of Mathematics},
    VOLUME = 168,
      YEAR = 2008,
     PAGES = {53--91},
      ISSN = {0021-2172,1565-8511},
       DOI = {10.1007/s11856-008-1056-4}
}

@incollection{ubi,
  author    = {Allouche, Jean-Paul and Shallit, Jeffrey},
  title     = {The ubiquitous {P}rouhet--{T}hue--{M}orse sequence},
  booktitle = {Sequences and Their Applications},
  series    = {Springer Ser. Discrete Math. Theor. Comput. Sci.},
  pages     = {1--16},
  publisher = {Springer-Verlag},
  address   = {London},
  year      = {1999}
}

@book{AS,
  author    = {Allouche, Jean-Paul and Shallit, Jeffrey},
  title     = {Automatic Sequences. Theory, Applications, Generalizations},
  publisher = {Cambridge Univ. Press},
  year      = {2003}
}

@book{Rigo_book,
    AUTHOR = {Rigo, Michel},
     TITLE = {Formal languages, automata and numeration systems. 1},
      NOTE = {Introduction to combinatorics on words},
 PUBLISHER = {ISTE, London; John Wiley \& Sons, Inc., Hoboken, NJ},
      YEAR = {2014},
     PAGES = {xx+301},
      ISBN = {978-1-84821-615-0}
}

@article{Cobham1972,
    AUTHOR = {Cobham, Alan},
     TITLE = {Uniform tag sequences},
   JOURNAL = {Math. Systems Theory},
  FJOURNAL = {Mathematical Systems Theory. An International Journal on
              Mathematical Computing Theory},
    VOLUME = {6},
      YEAR = {1972},
     PAGES = {164--192},
      ISSN = {0025-5661},
   MRCLASS = {94A30},
  MRNUMBER = {457011},
MRREVIEWER = {A.\ N.\ Maslov},
       DOI = {10.1007/BF01706087},
       URL = {https://doi.org/10.1007/BF01706087},
}

@article {MorseHedlund,
    AUTHOR = {Morse, Marston and Hedlund, Gustav A.},
     TITLE = {Symbolic {D}ynamics},
   JOURNAL = {Amer. J. Math.},
  FJOURNAL = {American Journal of Mathematics},
    VOLUME = {60},
      YEAR = {1938},
    NUMBER = {4},
     PAGES = {815--866},
      ISSN = {0002-9327,1080-6377},
       DOI = {10.2307/2371264}
}

@article {MorseHedlund2,
    AUTHOR = {Morse, Marston and Hedlund, Gustav A.},
     TITLE = {Symbolic dynamics {II}. {S}turmian trajectories},
   JOURNAL = {Amer. J. Math.},
  FJOURNAL = {American Journal of Mathematics},
    VOLUME = {62},
      YEAR = {1940},
     PAGES = {1--42},
      ISSN = {0002-9327,1080-6377},
       DOI = {10.2307/2371431}
}

@book {Lothaire2,
    AUTHOR = {Lothaire, M.},
     TITLE = {Algebraic combinatorics on words},
    SERIES = {Encyclopedia of Mathematics and its Applications},
    VOLUME = {90},
 PUBLISHER = {Cambridge University Press, Cambridge},
      YEAR = {2002},
     PAGES = {xiv+504},
      ISBN = {0-521-81220-8},
       DOI = {10.1017/CBO9781107326019}
}

@article {Adamczewski,
    AUTHOR = {Adamczewski, Boris},
     TITLE = {Balances for fixed points of primitive substitutions},
   JOURNAL = {Theoret. Comput. Sci.},
    VOLUME = {307},
      YEAR = {2003},
    NUMBER = {1},
     PAGES = {47--75},
        DOI = {10.1016/S0304-3975(03)00092-6}
}

@book {BaakeGrimm,
    AUTHOR = {Baake, Michael and Grimm, Uwe},
     TITLE = {Aperiodic order. {V}ol. 1},
    SERIES = {Encyclopedia of Mathematics and its Applications},
    VOLUME = {149},
 PUBLISHER = {Cambridge University Press, Cambridge},
      YEAR = {2013},
     PAGES = {xvi+531},
      ISBN = {978-0-521-86991-1},
   MRCLASS = {52C23 (11H06 20C35 20H15 82D25)},
  MRNUMBER = {3136260},
MRREVIEWER = {Jean-Pierre\ Gazeau},
       DOI = {10.1017/CBO9781139025256},
       URL = {https://doi.org/10.1017/CBO9781139025256},
}

@book {Queffelec,
    AUTHOR = {Queff\'elec, Martine},
     TITLE = {Substitution dynamical systems---spectral analysis},
    SERIES = {Lecture Notes in Mathematics},
    VOLUME = {1294},
   EDITION = {Second},
 PUBLISHER = {Springer-Verlag, Berlin},
      YEAR = {2010},
     PAGES = {xvi+351},
      ISBN = {978-3-642-11211-9},
       DOI = {10.1007/978-3-642-11212-6}
}

@incollection {Lepisto,
    AUTHOR = {Lepist\"o, Arto},
     TITLE = {On the power of periodic iteration of morphisms},
 BOOKTITLE = {Automata, languages and programming ({L}und, 1993)},
    SERIES = {Lecture Notes in Comput. Sci.},
    VOLUME = {700},
     PAGES = {496--506},
 PUBLISHER = {Springer, Berlin},
      YEAR = {1993},
      ISBN = {3-540-56939-1},
   MRCLASS = {68Q45 (68R15)},
  MRNUMBER = {1252429},
       DOI = {10.1007/3-540-56939-1\_97},
       URL = {https://doi.org/10.1007/3-540-56939-1_97},
}

@article {RigSti1,
    AUTHOR = {Rigo, Michel and Stipulanti, Manon},
     TITLE = {Automatic sequences: from rational bases to trees},
   JOURNAL = {Discrete Math. Theor. Comput. Sci.},
    VOLUME = {24},
      YEAR = {2022},
    NUMBER = {1},
     PAGES = {Paper No. 25, 26},
      ISSN = {1365-8050},
       DOI = {10.37236/6581}
}

@incollection {DekkingReg,
    AUTHOR = {Dekking,  Frederik Michel},
     TITLE = {Regularity and irregularity of sequences generated by
              automata},
 BOOKTITLE = {Seminar on {N}umber {T}heory, 1979--1980 ({F}rench)},
     PAGES = {Exp. No. 9, 10},
 PUBLISHER = {Univ. Bordeaux I, Talence},
      YEAR = {1980}
}

@incollection {CulikLepisto,
    AUTHOR = {Culik, Karel II and Karhum\"aki,  Juhani and Lepist\"o, Arto},
     TITLE = {Alternating iteration of morphisms and the {K}olakovski
              sequence},
 BOOKTITLE = {Lindenmayer systems},
     PAGES = {93--106},
 PUBLISHER = {Springer, Berlin},
      YEAR = {1992},
      ISBN = {3-540-55320-7},
   MRCLASS = {68Q45 (68R15)},
  MRNUMBER = {1226686},
       DOI = {10.1007/978-3-642-58117-5\_4},
       URL = {https://doi.org/10.1007/978-3-642-58117-5_4},
}

@inproceedings {Endrullis,
    AUTHOR = {Endrullis, J\"org and Hendriks, Dimitri},
     TITLE = {On periodically iterated morphisms},
 BOOKTITLE = {Proceedings of the {J}oint {M}eeting of the {T}wenty-{T}hird
              {EACSL} {A}nnual {C}onference on {C}omputer {S}cience {L}ogic
              ({CSL}) and the {T}wenty-{N}inth {A}nnual {ACM}/{IEEE}
              {S}ymposium on {L}ogic in {C}omputer {S}cience ({LICS})},
     PAGES = {Article No. 39, 10},
 PUBLISHER = {ACM, New York},
      YEAR = {2014},
      ISBN = {978-1-4503-2886-9},
   MRCLASS = {68Q42 (68R15)},
  MRNUMBER = {3397660},
       DOI = {10.1145/2603088.2603152},
       URL = {https://doi.org/10.1145/2603088.2603152},
}

@book {SGNT,
     TITLE = {Sequences, groups, and number theory},
    SERIES = {Trends in Mathematics},
    EDITOR = {Berth\'e, Val\'erie and Rigo, Michel},
 PUBLISHER = {Birkh\"auser/Springer, Cham},
      YEAR = {2018},
     PAGES = {xxvi+576},
      ISBN = {978-3-319-69152-7; 978-3-319-69151-0},
       DOI = {10.1007/978-3-319-69152-7}
}

@article {Jamet24,
    AUTHOR = {Boisson, Chlo\'e{} and Jamet, Damien and Marcovici, Ir\`ene},
     TITLE = {On a probabilistic extension of the {O}ldenburger-{K}olakoski
              sequence},
   JOURNAL = {RAIRO Theor. Inform. Appl. (RAIRO:ITA)},
  FJOURNAL = {RAIRO Theoretical Informatics and Applications (RAIRO: ITA)},
    VOLUME = 58,
      YEAR = 2024,
     PAGES = {Paper No. 11, 14},
      ISSN = {2804-7346},
   MRCLASS = {11K31 (11K36)},
  MRNUMBER = 4722165,
MRREVIEWER = {Vilius\ Stakenas},
       DOI = {10.1051/ita/2024005},
       URL = {https://doi.org/10.1051/ita/2024005},
}

@article{JacobsKeane1969,
    AUTHOR = {Jacobs, Konrad and Keane, Michael},
     TITLE = {{$0-1$}-sequences of {T}oeplitz type},
   JOURNAL = {Z. Wahrscheinlichkeitstheorie und Verw. Gebiete},
  FJOURNAL = {Zeitschrift f\"ur Wahrscheinlichkeitstheorie und Verwandte
              Gebiete},
    VOLUME = {13},
      YEAR = {1969},
     PAGES = {123--131},
   MRCLASS = {28.70},
       DOI = {10.1007/BF00537017}
}

@article {Sing,
    AUTHOR = {Sing, Bernd},
     TITLE = {More {K}olakoski sequences},
   JOURNAL = {Integers},
  FJOURNAL = {Integers. Electronic Journal of Combinatorial Number Theory},
    VOLUME = {11B},
      YEAR = {2011},
     PAGES = {Paper No. A14, 17},
      ISSN = {1553-1732},
   MRCLASS = {68R15 (37B10)},
  MRNUMBER = {3054433},
MRREVIEWER = {Takao\ Komatsu},
}

@article {Kimberling,
    AUTHOR = {Kimberling, Clark},
     TITLE = {Problems and {S}olutions: {A}dvanced {P}roblems: 6281},
   JOURNAL = {Amer. Math. Monthly},
  FJOURNAL = {American Mathematical Monthly},
    VOLUME = {86},
      YEAR = {1979},
    NUMBER = {9},
     PAGES = {793},
      ISSN = {0002-9890,1930-0972},
   MRCLASS = {99-04},
  MRNUMBER = {1539186},
       DOI = {10.2307/2322043},
       URL = {https://doi.org/10.2307/2322043},
}

@incollection {Keane91,
    AUTHOR = {Keane, Michael S.},
     TITLE = {Ergodic theory and subshifts of finite type},
 BOOKTITLE = {Ergodic theory, symbolic dynamics, and hyperbolic spaces
              ({T}rieste, 1989)},
    SERIES = {Oxford Sci. Publ.},
     PAGES = {35--70},
 PUBLISHER = {Oxford Univ. Press, New York},
      YEAR = {1991},
      ISBN = {0-19-853390-X; 0-19-859685-5},
   MRCLASS = {28D05},
  MRNUMBER = {1130172},
}

@article {Becher,
    AUTHOR = {Becher, Ver\'onica and Yuhjtman, Sergio A.},
     TITLE = {On absolutely normal and continued fraction normal numbers},
   JOURNAL = {Int. Math. Res. Not. IMRN},
  FJOURNAL = {International Mathematics Research Notices. IMRN},
      YEAR = {2019},
    NUMBER = {19},
     PAGES = {6136--6161},
      ISSN = {1073-7928,1687-0247},
       DOI = {10.1093/imrn/rnx297}
}

@article {Cassaigne-Karhumaki-1997,
    AUTHOR = {Cassaigne, Julien and Karhum\"aki, Juhani},
     TITLE = {Toeplitz words, generalized periodicity and periodically
              iterated morphisms},
   JOURNAL = {European J. Combin.},
  FJOURNAL = {European Journal of Combinatorics},
    VOLUME = {18},
      YEAR = {1997},
    NUMBER = {5},
     PAGES = {497--510},
      ISSN = {0195-6698,1095-9971},
   MRCLASS = {68R15 (11B85)},
  MRNUMBER = {1455183},
MRREVIEWER = {Jean-Paul\ Allouche},
       DOI = {10.1006/eujc.1996.0110},
       URL = {https://doi.org/10.1006/eujc.1996.0110},
}

@article {Allouche-Bacher-1992,
    AUTHOR = {Allouche, Jean-Paul and Bacher, Roland},
     TITLE = {Toeplitz sequences, paperfolding, {T}owers of {H}anoi and
              progression-free sequences of integers},
   JOURNAL = {Enseign. Math. (2)},
  FJOURNAL = {L'Enseignement Math\'ematique. Revue Internationale. 2e
              S\'erie},
    VOLUME = {38},
      YEAR = {1992},
    NUMBER = {3-4},
     PAGES = {315--327},
      ISSN = {0013-8584},
   MRCLASS = {11B85 (28D05)},
  MRNUMBER = {1189010},
MRREVIEWER = {G\'erald\ Tenenbaum},
}

@misc{Rao-Kola,
      title={Trucs et bidules sur la séquence de {K}olakoski}, 
      author={Rao,  Michael},
      year={2012},
      note={Available online at https://www.arthy.org/kola/kola.php (in French)}, 
}

@misc{Tao-Note,
      title={245C,  Notes 2: The {F}ourier transform}, 
      author={Tao,  Terence},
      year={2009},
      note={Available online at https://terrytao.wordpress.com/2009/04/06/the-fourier-transform/comment-page-2/}, 
}

@inproceedings {AJLTC-2024,
    AUTHOR = {Akiyama, Shigeki and Jamet, Damien and Marcovici, Ir\`ene and
              Tr\^an C\^ong, Mai-Linh},
     TITLE = {Self-descriptive sequences directed by two periodic sequences},
 BOOKTITLE = {Proceedings of the 13th edition of the conference on {R}andom
              {G}eneration of {C}ombinatorial {S}tructures. {P}olyominoes
              and {T}ilings},
    SERIES = {Electron. Proc. Theor. Comput. Sci. (EPTCS)},
    VOLUME = {403},
     PAGES = {18--22},
      YEAR = {2024},
}

@book{auslander,
  author    = {Auslander, Joseph},
  title     = {Minimal Flows and Their Extensions},
  series    = {North-Holland Mathematics Studies},
  volume    = {153},
  publisher = {North-Holland},
  address   = {Amsterdam},
  year      = {1988},
  doi = {10.1017/s0143385700005770},
}

\end{document}